\documentclass[10pt,reqno]{amsart}
\pdfoutput=1
\usepackage{amsmath} 
\usepackage{amssymb}
\usepackage{dsfont}
\usepackage[T1]{fontenc}
\usepackage{fancyhdr}
\usepackage{url}
\usepackage[foot]{amsaddr}

\usepackage[bordercolor=white, color=white]{todonotes}

\usepackage{color}
\usepackage{graphicx}

\def\squarebox#1{\hbox to #1{\hfill\vbox to #1{\vfill}}}

\newtheorem{Thm}{Theorem}[section]
 \newtheorem{cor}{Corollary}[section]
\newtheorem{Def}{Definition}[section]
\newtheorem{lem}{Lemma}[section]

\numberwithin{equation}{section}

\newcommand{\p}{\partial}
\newcommand{\bel}{\begin{equation} \label}
\newcommand{\ee}{\end{equation}}

\newcommand{\R}{\mathbb{R}}
\newcommand{\N}{\mathbb{N}}

\def\epsilon{\varepsilon}
\def\phi {\varphi}

\providecommand{\abs}[1]{\left\lvert#1\right\rvert}
\providecommand{\norm}[1]{\left\lVert#1\right\rVert}
\numberwithin{equation}{section}

\renewcommand{\leq}{\leqslant}
\renewcommand{\geq}{\geqslant}
\providecommand{\abs}[1]{\left\lvert#1\right\rvert}
\providecommand{\norm}[1]{\left\lVert#1\right\rVert}

\def\beq{\begin{equation}}
\def\eeq{\end{equation}}
\newcommand{\bea}{\begin{eqnarray}}
\newcommand{\eea}{\end{eqnarray}}
\newcommand{\beas}{\begin{eqnarray*}}
\newcommand{\eeas}{\end{eqnarray*}}

{

\DeclareMathOperator{\supp}{supp}
\DeclareMathOperator{\diam}{diam}
\DeclareMathOperator{\dist}{dist}
\newcommand{\I}{\mathds{1}}

\begin{document}

\title[Application of boundary control method ]{Application of the boundary control method to partial data Borg-Levinson inverse spectral problem}
\author{Y. Kian$^1$}
\address{$^1$Aix Marseille Univ, Universit\'e de Toulon, CNRS, CPT, Marseille, France}

\author{M. Morancey$^2$}
\address{$^2$Aix Marseille Univ, CNRS, Centrale Marseille, I2M, Marseille, France}

\author{L. Oksanen$^3$}
\address{$^3$Department of Mathematics, University College London, London, UK}

\begin{abstract}

We  consider  the multidimensional Borg-Levinson problem of determining a  potential $q$, appearing in the Dirichlet  realization of the  Schr\" odinger operator $A_q=-\Delta+q$ on a bounded domain $\Omega\subset\R^n$, $n\geq2$, from  the boundary spectral data of $A_q$ on an arbitrary portion of $\p\Omega$. More precisely, for $\gamma$ an open and non-empty subset of $\p\Omega$, we consider  the boundary spectral data on $\gamma$ given by  $\mathrm{BSD}(q,\gamma):=\{(\lambda_{k},{\partial_\nu \phi_{k}}_{|\overline{\gamma}}):\ k \geq1\}$, where $\{ \lambda_k:\ k \geq1\}$ is the non-decreasing sequence of eigenvalues of $A_q$,  $\{ \phi_k:\ k \geq1 \}$ an associated Hilbertian basis of eigenfunctions, and $\nu$ is the unit outward normal vector to $\partial\Omega$. We prove that the data $\mathrm{BSD}(q,\gamma)$ uniquely determine a bounded potential $q\in L^\infty(\Omega)$. Previous uniqueness results, with arbitrarily small $\gamma$, assume that $q$ is smooth. 
Our approach is based on the Boundary Control method,
and we give a self-contained presentation of the method, focusing on the analytic rather than geometric aspects of the method. 


\end{abstract}

\maketitle

\section{Introduction}
\subsection{Statement of the results}
We fix $\Omega$ a $\mathcal C^2$ bounded and connected domain  of $\R^n$, $n\geq2$, and $\gamma$ a non empty open set of $\Gamma=\partial\Omega$. We consider the Schr\"odinger operator $A_q=-\Delta_x+q$ acting on $L^2(\Omega)$ with Dirichlet boundary condition and $q\in L^\infty(\Omega)$ real valued. The spectrum of $A_q$ consists of a non decreasing sequence of eigenvalues $\{\lambda_k:\ k\in\N^*\}$, with $\N^*:=\{1,2,\ldots\}$, to which we associate a Hilbertian basis of eigenfunctions $\{\phi_k:\ k\in\N^*\}$. Then, we introduce the boundary spectral data restricted to the portion $\gamma$ given by 
$$\mathrm{BSD}(q,\gamma):=\left\{(\lambda_k,{\partial_\nu \phi_k}_{|\gamma}):\ k\in\N^*\right\},$$
with $\nu$ the outward unit normal vector to $\Gamma$ and $\p_\nu$ the normal derivative.
The main goal of the present paper is to prove uniqueness in the recovery of $q$ from the data $\mathrm{BSD}(q,\gamma)$.

\begin{Thm}\label{t1} Assume that $\Omega$ is convex, $\gamma \subset \Gamma$ is open and non empty, and $q_j\in L^\infty(\Omega)$, $j=1,2$. Then $\mathrm{BSD}(q_1,\gamma) = \mathrm{BSD}(q_2,\gamma)$
implies $q_1=q_2$.
\end{Thm}
This result will be proved by applying the so called Boundary Control method that we adapt to the present setting with a convex domain and a bounded potential. 

Let us also formulate a dynamic variant of Theorem \ref{t1}.
Fix $\Sigma=(0,T)\times\partial\Omega$, $Q=(0,T)\times\Omega$ with $0<T<\infty$, and consider the initial boundary value problem (IBVP in short)
\begin{equation}\label{eq1}
\left\{\begin{array}{ll}\partial_t^2u-\Delta_x u+q(x)u=0,\quad &\textrm{in}\ Q,\\  u(0,\cdot)=0,\quad \partial_tu(0,\cdot)=0,\quad &\textrm{in}\ \Omega,\\ u=f,\quad &\textrm{on}\ \Sigma.\end{array}\right.\end{equation}
According to \cite[Theorem 2.1]{LLT}, for $f\in H^1(\Sigma)$, the problem \eqref{eq1} admits a unique solution 
$$u\in\mathcal C([0,T];H^1(\Omega))\cap \mathcal C^1([0,T];L^2(\Omega))$$ which satisfies $\p_\nu u\in L^2(\Sigma)$.
Thus we can define the partial Dirichlet-to-Neumann map 
$$\Lambda(q,\gamma,T):\mathcal C^\infty_0((0,T)\times\gamma)\ni f\mapsto\p_\nu u_{|(0,T)\times\gamma}.$$
We define also $\diam(\Omega) = \max \{|x-y|:\ x,y \in \overline \Omega\}$.
The dynamic variant can be stated in the following way
\begin{Thm}\label{tt1} Assume that $\Omega$ is convex, $\gamma \subset \Gamma$ is open and non empty, $T>2\diam(\Omega)$, and $q_j\in L^\infty(\Omega)$, $j=1,2$. Then $\Lambda(q_1,\gamma,T) =\Lambda(q_2,\gamma,T)$
implies that $q_1=q_2$.
\end{Thm}

\subsection{Previous literature}
Our problem   is a generalization to the multidimensional case of the pioneering work of Borg \cite{Bo}, Levinson \cite{L}, Gel'fand and Levitan  \cite{GL} stated in an interval of $\R$, also called Borg-Levinson inverse spectral problem. The first multidimensional formulation of this problem is given by Nachman, Sylvester and Uhlmann \cite{NSU} who applied the result of \cite{SU} to prove that $B(q,\p\Omega)$ determines uniquely $q$. P\"aiv\"arinta and Serov \cite{PS} extended this result to $q \in L^p$, and Canuto and Kavian \cite{CK2} to more general perturbations of the Laplacian. Isozaki~\cite{I} proved that the uniqueness still holds if finitely many eigenpairs remain unknown and \cite{CS,KKS,Ki} proved that only some asymptotic knowledge of $B(q,\p\Omega)$ is enough for the recovery of $q$ as well as more general coefficients. 

Let us now turn to partial data results. 
For arbitrarily small $\gamma$, the known uniqueness results are based on
the Boundary Control method introduced by Belishev \cite{B}.
In \cite{KK}, under the assumption that $q$ is smooth, Katchalov and Kurylev proved that the data $B(q,\p\Omega)$, with the exception of finitely many eigenpairs, 
determines $q$,
and \cite{KKL} proved that the uniqueness remains true when knowing only the partial boundary spectral  data $B(q,\gamma)$, with $\gamma$ an arbitrary portion of the boundary. 
The novelty of the present paper is to consider non-smooth $q$.

Let us remark that more general operators than 
the Schr\"odinger operator have been considered. 
It was proved in \cite{BK} that, when $\Omega$ is a smooth Riemanian manifold the boundary, the spectral data $B(0,\p\Omega)$ determines the Riemanian manifold up to an isometry. 
Moreover, arbitrary smooth and symmetric lower order perturbations of the Laplace-Beltrami operator can be determined up to natural gauge transformations, see \cite{Ku} and, for the case of equations taking values on Hermitian vector bundles, \cite{KOP}. These results allow $\gamma$ to be arbitrarily small. 
It is an open question, however, if the recovery of non-symmetric lower order perturbations is possible without further geometric assumptions. 
The known results \cite{KL} assume that $\gamma$ satisfies the geometric control condition \cite{BLR}.

All the results of the present paper can be extended to the recovery of more general coefficients on a smooth Riemannian manifold, by changing some intermediate tools and by replacing the last part of the proof, that is, the global recovery step, with the iterative process described in \cite[Section 4.2]{KOP}. The assumption of convexity allows us to simplify in various way the exposition in order to emphasize the main idea, and analytic aspects, of the Boundary Control method. The geometric aspects are mostly avoided, since for a pair points on a convex domain,  the shortest path between the points is simply a line segment. 
For these reasons the present paper can also be considered as an introduction to the Boundary Control method.

The dynamic variant in Theorem \ref{tt1} allows for a more fine grained notion partial data where $f$ is supported on a part of boundary, disjoint from the part on which $\p_\nu u$ is restricted. Such disjoint data questions have been studied in \cite{LO,LO2}, however, the techniques used the present paper do not readily extend to disjoint data cases.

\subsection{Outline}
In Section 2 we recall some properties of solutions of \eqref{eq1} that will be used in the proof of Theorem \ref{t1}. In Section 3 we 
describe the Boundary Control method and use it to
show that $q$ can be recovered locally near $\gamma$. Building on the local recovery step, we show in Section 4 the global recovery as stated in Theorem \ref{t1}.
In Section 5 we show how to prove Theorem \ref{tt1} by adapting the proof of Theorem \ref{t1}. 
For the convenience of the reader, we prove in the appendix some well-known facts formulated in Section~2.

\section{Finite speed of propagation and unique continuation}
\label{Sec:WavesResults}

The Boundary Control method is based on two complementary properties of the wave equation (\ref{eq1}): the finite speed of propagation and unique continuation. Loosely speaking, they give respectively the 
maximum and minimum speeds at which waves can propagate. It is essential for the Boundary Control method that these two speeds are the same
in the case of a scalar valued wave equations such as (\ref{eq1}).
All the results recalled in this section are well-known, however, for the convenience of the reader, we give their proofs in the appendix. 

\subsection{Finite speed of propagation and domains of influence}

We make the standing assumption that $\Omega$ is convex and define 
$$
\textrm{dist}(x,S) = \inf \{|x-y|:\ y \in S\}, \quad x \in \overline \Omega,\ S \subset \overline \Omega.
$$
We write also $B(x,r) = \{y \in \overline \Omega:\ |y-x| < r \}$,
$x \in \overline \Omega$, $r > 0$.
A typical formulation of the finite speed of propagation is as follows
\begin{lem}\label{l1} 
Let $q\in L^\infty(\Omega)$, let $x \in \overline \Omega$, and let $\tau>0$. Define the cone 
$$D=\{(t,y)\in [0,\tau]\times\overline{\Omega}:\ |y-x| < \tau - t\},$$
and consider $u\in \mathcal C([0,\tau];H^1(\Omega))\cap \mathcal C^1([0,\tau];L^2(\Omega))$ satisfying  $(\p_t^2-\Delta_x+q)u=0$ in $(0,\tau)\times\Omega$. Then
$$
u_{|\{0\} \times B(x,\tau)} = \p_t u_{|\{0\} \times B(x,\tau)} = 0,
\quad  u_{|D\cap \Sigma}=0,
$$
imply $u_{|D}=0$.
\end{lem}

A proof of this classical result can be found e.g. in \cite[Theorem 2.47]{KKL}. 
Let us now reformulate Lemma \ref{l1} by using the notion of domain of influence.

\begin{Def}\label{d1} For every $\tau>0$ and every open subset $S$ of $\Gamma$ we define the subset $\Omega(S,\tau)$ of $\overline{\Omega}$ given by
$$\Omega(S,\tau)=\{x\in\overline{\Omega}:\ \dist(x,S)\leq \tau\}.$$
The set $\Omega(S,\tau)$ is called the domain of influence of $S$ at time $\tau$.\end{Def}

\begin{Thm}\label{t2} 
Let $S$ be an open subset of $\Gamma$ and $\tau\in(0,T]$.
Let $u$ solve \eqref{eq1} with $f\in H^1(\Sigma)$ satisfying $\supp(f)\subset (0,T]\times S$. Then $\supp[u(\tau,\cdot)]\subset \Omega(S,\tau)$. 
\end{Thm}

We give a proof of this theorem in the appendix. The proof is a short reduction to Lemma \ref{l1}.

\subsection{Unique continuation and approximate controllability}

The local unique continuation result \cite{T1} by Tataru 
implies the following global Holmgren-John type unique continuation

\begin{Thm}\label{t3} Let $q\in L^\infty(\Omega)$, let $S$ be an open subset of $\Gamma$, and let $\tau>0$. Consider $$u\in \mathcal C([0,2\tau];H^1(\Omega))\cap \mathcal C^1([0,2\tau];L^2(\Omega))$$ satisfying  $(\p_t^2-\Delta_x+q)u=0$ in $(0,2\tau)\times\Omega$. Then 
\bel{t3a}u_{|[0,2\tau]\times S}=\partial_\nu u_{|[0,2\tau]\times S}=0\ee
implies
$u(\tau,x)=\p_tu(\tau,x)=0$, $x\in \Omega(S,\tau)$. 
\end{Thm}

We give a proof of this theorem and the following corollary in the appendix. The corollary is often called approximate controllability.
We denote by $u^f = u$ the solution of \eqref{eq1} when emphasizing the dependence on the boundary source $f$.

\begin{cor}\label{c1} Let $S$ be an open subset of $\Gamma$ and $\tau\in(0,T]$.
Then the set
\begin{equation}\label{c1a} \{u^f(\tau,\cdot):\ f\in \mathcal C^\infty_0((0,\tau) \times S)\}\end{equation}
is dense in $L^2(\Omega(S,\tau))$.
\end{cor}

Here $L^2(\Omega(S,\tau))$ is considered as the subspace 
of $L^2(\Omega)$ consisting of functions vanishing outside $\Omega(S,\tau)$. 
Note that Theorem \ref{t2} implies that the set \eqref{c1a} is indeed contained in the subspace $L^2(\Omega(S,\tau))$, and in this sense,
Corollary \ref{c1} relates the finite speed of propagation and unique continuation. The Boundary Control method depends heavily on this relation, as described in the next section.

\section{Local recovery of the potential}

We make the following standing assumption
\begin{itemize}
\item[(A)] $\Omega$ is convex, $\gamma \subset \Gamma$ is open and non-empty, and $q_j\in L^\infty(\Omega)$, $j=1,2$,
\end{itemize}
and write $B_{q_j} = \mathrm{BSD}(q_j,\gamma)$. 
In this section we prove the following 

\begin{Thm}\label{t5} 
Suppose that $B_{q_1} = B_{q_2}$.
Then there are $\tau\in(0,T)$ and a non-empty open set $\gamma'\subset\gamma$ such that 
\begin{equation}\label{t5a} q_1(x)=q_2(x),\quad x\in \Omega(\gamma',\tau).
\end{equation}
\end{Thm}

\subsection{From boundary spectral data to inner products of solutions}
We write $u_j^f$ for the solution of (\ref{eq1}) with $q=q_j$, $j=1,2$, and $T = 2\diam(\Omega) + 1$. 
Moreover, we denote by $\phi_{j,k}$, $k \in \N^*$, a fixed Hilbertian basis of eigenfunctions of $A_{q_j}$, $j=1,2$.
Let us begin by showing that the Fourier coefficients of $u_j^f(t) := u_j^f(t,\cdot)$, with respect to the bases $\phi_{j,k}$, $k \in \N^*$, coincide for $j=1$ and $j=2$.

\begin{lem}\label{l4} 
Suppose that $B_{q_1} = B_{q_2}$. 
Let $f\in H^1(\Sigma)$ satisfy $\supp(f)\subset (0,T]\times \gamma$. 
Then
\begin{align}\label{l4a}
\left\langle u_1^f(t),\phi_{1,k}\right\rangle_{L^2(\Omega)}=\left\langle u_2^f(t),\phi_{2,k}\right\rangle_{L^2(\Omega)}, \quad k \in \N^*, t\in [0,T].
\end{align}
\end{lem}

\begin{proof} 
Write $B_{q_j} = \left\{(\lambda_k,\psi_k):\ k\in\N^*\right\}$, $j=1,2$.
We start by assuming that $f\in\mathcal C^\infty_0((0,T]\times\gamma)$. Then $u_j^f\in H^2(Q)$. Setting
$v_{j,k}(t):=\left\langle u_j^f(t),\phi_{j,k}\right\rangle_{L^2(\Omega)}$
and integrating by parts we find
$$v_{j,k}''(t)=\left\langle \p_t^2u_j^f(t),\phi_{j,k}\right\rangle_{L^2(\Omega)}=-\left\langle (-\Delta_x+q_j)u_j^f(t),\phi_{j,k}\right\rangle_{L^2(\Omega)}=-\lambda_{k}v_{j,k}(t)+\int_{\p\Omega}f(t,x)\overline{\p_\nu\phi_{j,k}(x)}d\sigma(x).$$
As $\supp(f)\subset (0,T]\times \gamma$, we deduce that both $v_{j,k}$, $j=1,2$, solve the same differential equation
\begin{align}\label{def_vk}
v_{k}''(t)+\lambda_kv_{k}(t)=\int_{\gamma}f(t,x)\overline{\psi_k(x)}d\sigma(x),\quad v_{k}(0)=v_{k}'(0)=0,
\end{align}
which implies \eqref{l4a}. By density, \eqref{l4a} holds also for $f\in H^1(\Sigma)$ satisfying $\supp(f)\subset (0,T]\times \gamma$.

\end{proof}

Under the assumptions of Lemma \ref{l4}, it holds in particular that 
\begin{align}
\label{id1}
\norm{u_1^f(t)}_{L^2(\Omega)}^2 &=\norm{u_2^f(t)}_{L^2(\Omega)}^2 = \sum_{k=1}^\infty |v_k^f(t)|^2, \quad t \in [0,T],
\\\label{id2}
\left\langle u_1^{f}(t),u_1^{g}(s)\right\rangle_{L^2(\Omega)}&=\left\langle u_2^{f}(t),u_2^{g}(s)\right\rangle_{L^2(\Omega)}=\sum_{k=1}^\infty v_k^f(t)\overline{v_k^g(s)}, \quad t, s \in [0,T],
\end{align}
where $v_k^f = v_k$ is the solution of (\ref{def_vk}).

\subsection{Inner products on domains of influence}
We denote by $\I_S$ the indicator function of a set $S$, that is, 
$\I_S(x) = 1$ if $x \in S$ and $\I_S(x) = 0$ otherwise.
Let us show that (\ref{id2}) holds when
$\Omega$ is replaced by a domain of influence $\Omega(\gamma',\tau)$
in the following sense

\begin{lem}
\label{l_innerp}
Suppose that $B_{q_1} = B_{q_2}$. 
Let $f,g\in H^1(\Sigma)$ supported in $(0,T]\times \gamma$,
let $\gamma' \subset \gamma$ be open, and let $\tau \in (0,T]$.
Then
\begin{align}\notag 
\left\langle \I_{\Omega(\gamma',\tau)} u_1^{f}(t),u_1^{g}(s)\right\rangle_{L^2(\Omega)}=\left\langle \I_{\Omega(\gamma',\tau)}  u_2^{f}(t),u_2^{g}(s)\right\rangle_{L^2(\Omega)}, \quad t,s\in [0,T].
\end{align}
\end{lem}
\begin{proof}
By Corollary \ref{c1}, there is a sequence $(f_k)_{k\in\N^*}$ in $\mathcal C_0^\infty((0,\tau) \times \gamma')$ such that $u_1^{f_k}(\tau)$ converges to $\I_{\Omega(\gamma',\tau)} u_1^{f}(t)$ in $L^2(\Omega)$ as $k \to +\infty$. As $\I_{\Omega(\gamma',\tau)} u_1^{f}(t)$ is the orthogonal projection of $u_1^{f}(t)$ into the subspace $L^2(\Omega(\gamma',\tau))$, it holds, using again the density in Corollary \ref{c1}, that 
$$
\lim_{k\to+\infty}\norm{u_1^{f_{k}}(\tau)-u_1^f(t)}_{L^2(\Omega)}=
\inf_{h\in\mathcal C^\infty_0((0,\tau)\times \gamma')}\norm{u_1^{h}(\tau)-u_1^f(t)}_{L^2(\Omega)}.
$$
Now (\ref{id2}) implies that 
\begin{equation*}
\lim_{k\to+\infty}\norm{u_2^{f_{k}}(\tau)-u_2^f(t)}_{L^2(\Omega)}=\lim_{k\to+\infty}\norm{u_1^{f_{k}}(\tau)-u_1^f(t)}_{L^2(\Omega)} =
\inf_{h\in\mathcal C^\infty_0((0,\tau)\times \gamma')}\norm{u_1^{h}(\tau)-u_1^f(t)}_{L^2(\Omega)}
\end{equation*}
and also that for any $h\in\mathcal C^\infty_0((0,\tau)\times \gamma')$
\begin{equation*}
\norm{u_1^{h}(\tau)-u_1^f(t)}_{L^2(\Omega)} =\norm{u_2^{h}(\tau)-u_2^f(t)}_{L^2(\Omega)}.
\end{equation*}
The last two equalities imply
\begin{equation} \label{u2fk_conv}
\lim_{k\to+\infty}\norm{u_2^{f_{k}}(\tau)-u_2^f(t)}_{L^2(\Omega)}=
\inf_{h\in\mathcal C^\infty_0((0,\tau)\times \gamma')}\norm{u_2^{h}(\tau)-u_2^f(t)}_{L^2(\Omega)}.
\end{equation}
The convergence of $(u_1^{f_k}(\tau))_{k\in\N^*}$ and (\ref{id1})
imply that $(u_2^{f_k}(\tau))_{k\in\N^*}$ is a Cauchy sequence, and therefore it converges. Corollary \ref{c1} and (\ref{u2fk_conv}) imply that
$u_2^{f_k}(\tau)$ converges to $\I_{\Omega(\gamma',\tau)} u_2^{f}(t)$. Finally, by (\ref{id2}),
\begin{align}\notag 
\left\langle \I_{\Omega(\gamma',\tau)} u_1^{f}(t),u_1^{g}(s)\right\rangle_{L^2(\Omega)}
&= \lim_{k\to+\infty}
\left\langle u_1^{f_k}(\tau),u_1^{g}(s)\right\rangle_{L^2(\Omega)}
= \lim_{k\to+\infty}
\left\langle u_2^{f_k}(\tau),u_2^{g}(s)\right\rangle_{L^2(\Omega)}
\\\notag&=\left\langle \I_{\Omega(\gamma',\tau)}  u_2^{f}(t),u_2^{g}(s)\right\rangle_{L^2(\Omega)}.
\end{align}

\end{proof}

From this result, we deduce the following corollary.

\begin{cor}
\label{c_innerp}
Suppose that $B_{q_1} = B_{q_2}$. 
Let $f,g\in H^1(\Sigma)$ be supported in $(0,T]\times \gamma$,
and let $x \in \gamma$ and $\tau \in (0,T]$.
Then

\bel{ccc}
\left\langle \I_{B(x,\tau)} u_1^{f}(t),u_1^{g}(s)\right\rangle_{L^2(\Omega)}=\left\langle \I_{B(x,\tau)}  u_2^{f}(t),u_2^{g}(s)\right\rangle_{L^2(\Omega)}, \quad t,s \in [0,T].\ee
\end{cor}
\begin{proof} For $\epsilon>0$, we fix $$\gamma_\epsilon:=\{y\in\gamma:|x-y|<\epsilon\},\quad \Omega_\epsilon=\Omega(\gamma_\epsilon,\tau)\setminus B(x,\tau).$$
For $j=1,2$, since $u_j^f\in\mathcal C([0,T];H^1(\Omega))$, by the Sobolev embedding theorem for
$$p:=\left\{\begin{array}{l} 3\quad \textrm{for }n=2\\ \frac{2n}{n-2}\quad \textrm{for }n\geq3\end{array}\right.$$
we have $u_j^f\in\mathcal C([0,T];L^p(\Omega))$. Thus, an application of the H\"older inequality yields
$$\abs{\left\langle \I_{\Omega_\epsilon} u_j^{f}(t),u_j^{g}(s)\right\rangle_{L^2(\Omega)}}\leq |\Omega_\epsilon|^{\frac{p-2}{2p}}\norm{u_j^{f}}_{\mathcal C([0,T];L^p(\Omega))}\norm{u_j^{g}}_{\mathcal C([0,T];L^2(\Omega))}.$$
Thus, for $j=1,2$, we have 
\bel{ccc1}\lim_{\epsilon\to0}\left(\left\langle \I_{\Omega(\gamma_\epsilon,\tau)} u_j^{f}(t),u_j^{g}(s)\right\rangle_{L^2(\Omega)}-\left\langle \I_{B(x,\tau)} u_1^{f}(t),u_1^{g}(s)\right\rangle_{L^2(\Omega)}\right)=\lim_{\epsilon\to0}\left\langle \I_{\Omega_\epsilon} u_j^{f}(t),u_j^{g}(s)\right\rangle_{L^2(\Omega)}=0.\ee
On the other hand, Lemma \ref{l_innerp} implies 
$$\left\langle \I_{\Omega(\gamma_\epsilon,\tau)} u_1^{f}(t),u_1^{g}(s)\right\rangle_{L^2(\Omega)}=\left\langle \I_{\Omega(\gamma_\epsilon,\tau)}  u_2^{f}(t),u_2^{g}(s)\right\rangle_{L^2(\Omega)}, \quad t,s\in [0,T], \epsilon>0.$$
Combining this with \eqref{ccc1} and sending $\epsilon\to0$, we deduce \eqref{ccc}.

\end{proof}
The proof of Lemma \ref{l_innerp} can be iterated as follows.

\begin{lem}
\label{l_iinnerp}
Suppose that $B_{q_1} = B_{q_2}$. 
Let $f,g\in H^1(\Sigma)$ be supported in $(0,T]\times \gamma$,
let $\gamma',\gamma'' \subset \gamma$ be open, and let $\tau',\tau'' \in (0,T]$.
Then
\begin{align}\notag 
\left\langle \I_{\Omega(\gamma',\tau')} \I_{\Omega(\gamma'',\tau'')} u_1^{f}(t),u_1^{g}(s)\right\rangle_{L^2(\Omega)}=\left\langle \I_{\Omega(\gamma',\tau')} \I_{\Omega(\gamma'',\tau'')}  u_2^{f}(t),u_2^{g}(s)\right\rangle_{L^2(\Omega)}, \quad t, s \in [0,T].
\end{align}
\end{lem}
\begin{proof}
By the proof of Lemma \ref{l_innerp},
there is a sequence $(f_k)_{k\in\N^*}$ in $\mathcal C_0^\infty((0,\tau'') \times \gamma'')$ such that for both $j=1$ and $j=2$,
the functions $u_j^{f_k}(\tau'')$ converge to $\I_{\Omega(\gamma'',\tau'')} u_j^{f}(t)$ in $L^2(\Omega)$ as $k \to +\infty$. 
Then by Lemma \ref{l_innerp},
\begin{align}\notag 
&\left\langle \I_{\Omega(\gamma',\tau')} \I_{\Omega(\gamma'',\tau'')} u_1^{f}(t),u_1^{g}(s)\right\rangle_{L^2(\Omega)}
= \lim_{k\to+\infty}
\left\langle \I_{\Omega(\gamma',\tau')} u_1^{f_k}(t),u_1^{g}(s)\right\rangle_{L^2(\Omega)}
\\\notag&\quad= \lim_{k\to+\infty}
\left\langle \I_{\Omega(\gamma',\tau')} u_2^{f_k}(t),u_2^{g}(s)\right\rangle_{L^2(\Omega)}
=\left\langle \I_{\Omega(\gamma',\tau')} \I_{\Omega(\gamma'',\tau'')}  u_2^{f}(t),u_2^{g}(s)\right\rangle_{L^2(\Omega)}.
\end{align}
\end{proof}

\subsection{Recovery of internal data near $\gamma$}

Let $x \in \Omega$ and let $y$ be one of the closest point in $\p \Omega$ to $x$.
Then the line through $x$ and $y$ must intersect $\p \Omega$ perpendicularly. Conversely, a point $y \in \p \Omega$ is the closest point in $\p \Omega$ to $x = y - r \nu(y)$ for small $r > 0$.
Here $\nu(y)$ is the outward unit normal vector at $y$.
Furthermore, there is $\tau_0>0$ such that 
$$
r = \dist(y - r \nu(y), \p \Omega), \quad r \in [0,\tau_0],\ y \in \p\Omega.
$$
We will show that Theorem \ref{t5} holds with any choice of $\tau > 0$ and $\gamma' \subset \gamma$ satisfying 
\bel{geo}\Omega(\gamma',\tau)\subset\{y-r\nu(y):\ r\in [0,\tau_0),\  y\in\gamma\}.\ee
This hypothesis as well as the set $N(\gamma)$ introduced in the following lemma are illustrated in Fig.~\ref{FigureCondGeo}.

\medskip
\begin{figure}[h]
\includegraphics[width=.5\textwidth]{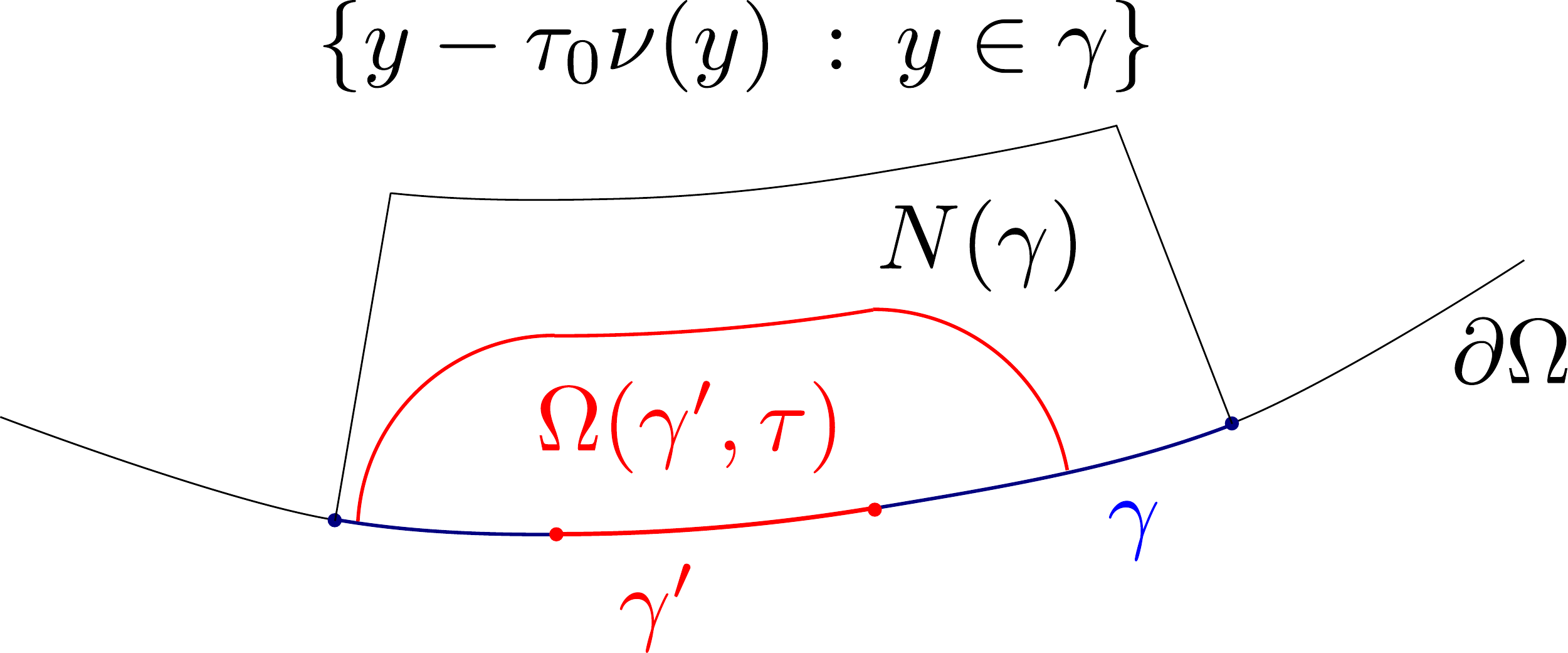}
\caption{\label{FigureCondGeo} Geometric condition~\eqref{geo}}
\end{figure}

We show next that inner products on domains of influence can be used to determine the following pointwise products 

\begin{lem}
\label{l_ptprod}
Suppose that $B_{q_1} = B_{q_2}$. 
Let $f,g\in H^1(\Sigma)$ be supported in $(0,T]\times \gamma$.
Then
\begin{align}\notag 
u_1^{f}(t,x) \overline{u_1^{g}(s,x)}
= u_2^{f}(t,x) \overline{u_2^{g}(s,x)}, \quad t,s \in [0,T],\ x \in N(\gamma),
\end{align}
where $N(\gamma) = \{y-r\nu(y):\ r\in(0,\tau_0),\  y\in\gamma\}$.
\end{lem}
\begin{proof}
To illustrate the idea of the proof, let us suppose for a moment that $q_j$, $j=1,2$, and $\Omega$ are smooth. Then for smooth $f$ and $g$, also the functions $u^f_j$ and $u^g_j$ are smooth.
Let $y \in \gamma$, $r \in (0,\tau_0)$, and set $x = y-r\nu(y)$, 
$\tilde A_{\epsilon,x} = B(y,r+\epsilon) \setminus \Omega(\gamma,r-\epsilon)$, $\epsilon > 0$. Then $\tilde A_{\epsilon,x} \to \{x\}$ as $\epsilon \to 0$. By taking a limit analogous to that in Corollary \ref{c_innerp}, it follows from Lemma \ref{l_iinnerp} that
\begin{align}\notag 
\left\langle \I_{B(y,r+\epsilon)} \I_{\Omega(\gamma,r-\epsilon)} u_1^{f}(t),u_1^{g}(s)\right\rangle_{L^2(\Omega)}=\left\langle \I_{B(y,r+\epsilon)} \I_{\Omega(\gamma,r-\epsilon)} u_2^{f}(t),u_2^{g}(s)\right\rangle_{L^2(\Omega)},
\end{align}
where $t, s \in [0,T]$ and $f,g\in \mathcal C_0^\infty((0,T)\times \gamma)$. Combining this with Corollary \ref{c_innerp}, we obtain
\begin{align}\notag 
\left\langle \I_{B(y,r+\epsilon)} (1-\I_{\Omega(\gamma,r-\epsilon)}) u_1^{f}(t),u_1^{g}(s)\right\rangle_{L^2(\Omega)}=\left\langle \I_{B(y,r+\epsilon)} (1-\I_{\Omega(\gamma,r-\epsilon)}) u_2^{f}(t),u_2^{g}(s)\right\rangle_{L^2(\Omega)},
\end{align}
and therefore, denoting the volume of $\tilde A_{\epsilon,x}$ by $|\tilde A_{\epsilon,x}|$,
\begin{align}\notag 
|\tilde A_{\epsilon,x}|^{-1} \left\langle \I_{\tilde A_{\epsilon,x}} u_1^{f}(t),u_1^{g}(s)\right\rangle_{L^2(\Omega)}= |\tilde A_{\epsilon,x}|^{-1} \left\langle \I_{\tilde A_{\epsilon,x}} u_2^{f}(t),u_2^{g}(s)\right\rangle_{L^2(\Omega)}.
\end{align}
Letting $\epsilon \to 0$, we obtain $u_1^{f}(t,x) \overline{u_1^{g}(t,x)}
= u_2^{f}(t,x) \overline{u_2^{g}(t,x)}$.

Let us now turn to the case of bounded $q_j$, $j=1,2$. The above argument does not generalize immediately, since the limit with respect to $\epsilon$ might not exist in the non-smooth case. 
Our remedy is to replace the sets $\tilde A_{\epsilon,x}$ with sets of bounded eccentricity. 

Let $x$ be as above. 
Choose unit vectors $\xi_1,\dots,\xi_n \in \R^n$, that form a basis of $\R^n$, and that are small enough perturbations of $-\nu(y)$ so that the lines $s \mapsto x + s \xi_j$ intersect $\partial \Omega$ in $\gamma$ near $y$. Denote the points of intersection by $z_j$, and consider the sets
\bel{Ax1}
A_{x,\epsilon} = B_{x,\epsilon}\setminus \Omega(\gamma,r-\epsilon), 
\quad 
B_{x,\epsilon} = \bigcap_{j=1}^nB(z_j,|x-z_j|+\epsilon), \quad \epsilon > 0.
\ee
The construction of the set $A_{x,\epsilon}$ is illustrated in Fig.~\ref{FigureExcentriciteFinie}.
\medskip
\begin{figure}[h]
\includegraphics[width=.5\textwidth]{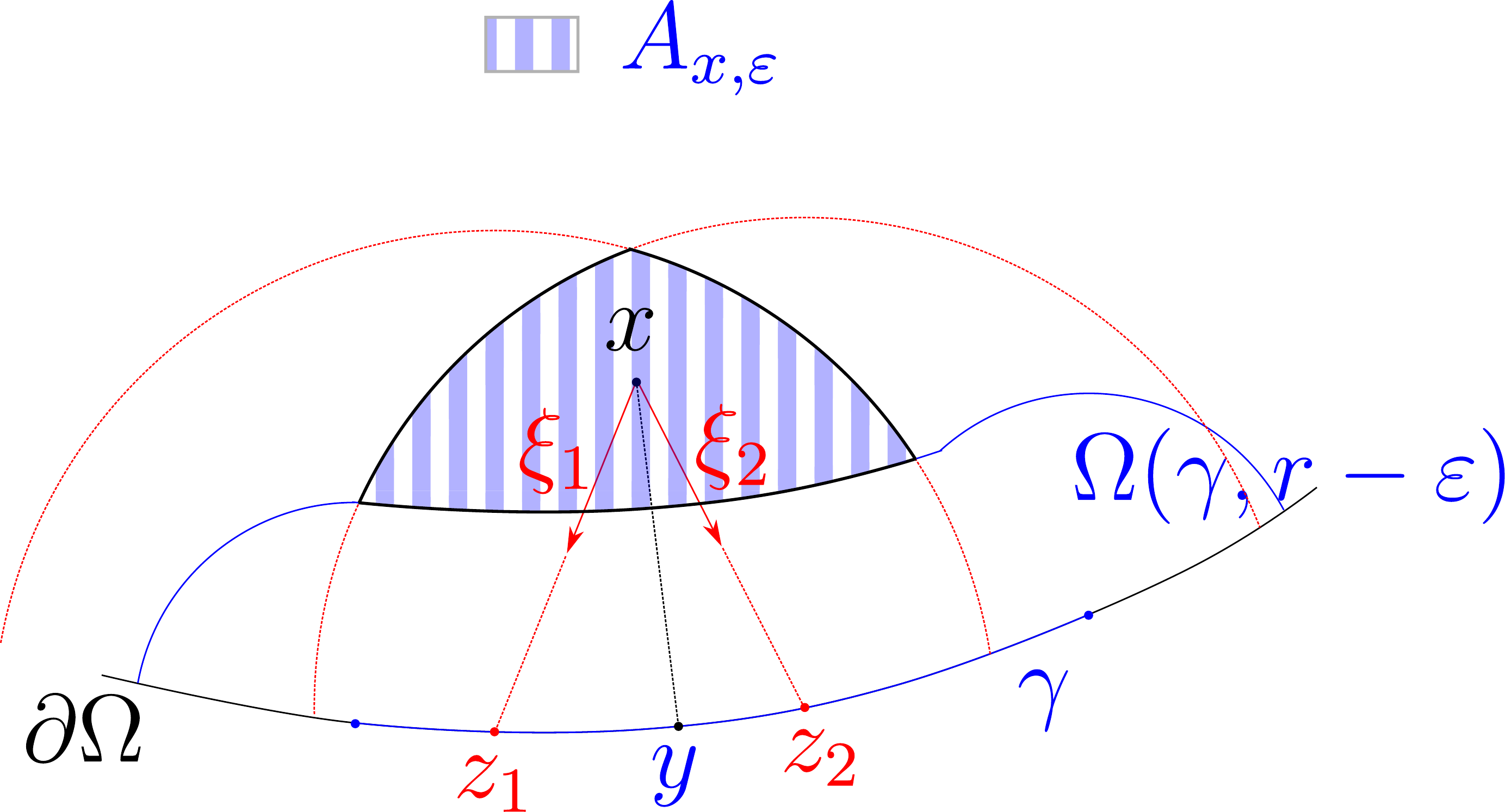}
\caption{\label{FigureExcentriciteFinie} Sets $A_{x,\epsilon}$}
\end{figure}

For small $\epsilon > 0$, the set $A_{x,\epsilon}$ is approximated in the first order by the simplex with outward normals $-\nu(y)$ and $\xi_j$, $j=1,\dots,n$, and all the faces having distance $\epsilon$ to $x$. Thus $A_{x,\epsilon}$ is of bounded eccentricity and $A_{\epsilon,x} \to \{x\}$ as $\epsilon \to 0$.

By repeating the proofs of Lemma \ref{l_iinnerp} and Corollary \ref{c_innerp} several times, we obtain analogously to the smooth case,
\begin{align}\notag 
|A_{\epsilon,x}|^{-1} \left\langle \I_{A_{\epsilon,x}} u_1^{f}(t),u_1^{g}(s)\right\rangle_{L^2(\Omega)}= |A_{\epsilon,x}|^{-1} \left\langle \I_{A_{\epsilon,x}} u_2^{f}(t),u_2^{g}(s)\right\rangle_{L^2(\Omega)}.
\end{align}
The Lebesgue differentiation theorem, see e.g. \cite[Chapter 7, Theorem 7.14]{R}, implies the claim. Note that the products in the claim are interpreted as $L^1$-functions.

\end{proof}

\begin{lem}
\label{l_pt}
Suppose that $B_{q_1} = B_{q_2}$. 
Let $f,g\in H^1(\Sigma)$ be supported in $(0,T]\times \gamma$.
Then
\begin{align}\notag 
u_1^{f}(t,x)
= u_2^{f}(t,x), \quad t \in [0,T],\ x \in N(\gamma).
\end{align}
\end{lem}
\begin{proof}
We will choose $u_j^g$ in Lemma \ref{l_ptprod} to be a suitable geometric optics solution, and begin by constructing such solutions. 
Let $y \in \gamma$, $r \in (0,\tau_0)$, and set $x_0 = y-r\nu(y)$.
Let $\delta > 0$ be small and set $s_1 = r + \delta$ and $s_2 = r + 2 \delta$.
The line $\beta(t) = x_0 + (s_1-t) \nu(y)$
satisfies $\beta(s_1) = x_0$, $\beta(s_1-r) 
= y \in \gamma$,
and $\beta(0) \in \R^n \setminus \Omega$. Hence if 
$\delta$ is small enough and 
$\chi \in \mathcal C_0^\infty(\R^n)$ has small enough support, then the function
$a(t,x) = \chi(x - \beta(t))$
satisfies
\begin{align}\label{suppa}
\supp(a) \cap ([0,s_2] \times \Omega) \subset (0,s_2] \times N(\gamma), \quad 
\supp(a) \cap ([0,s_2] \times \p \Omega) \subset (0,s_2) \times \gamma.
\end{align}
In particular, $\supp[a(0,\cdot)] \subset \R^n \setminus \Omega$. The support of this particular solution is illustrated in Fig.~\ref{FigureSupport}.


\medskip
\begin{figure}[h]
\includegraphics[width=.5\textwidth]{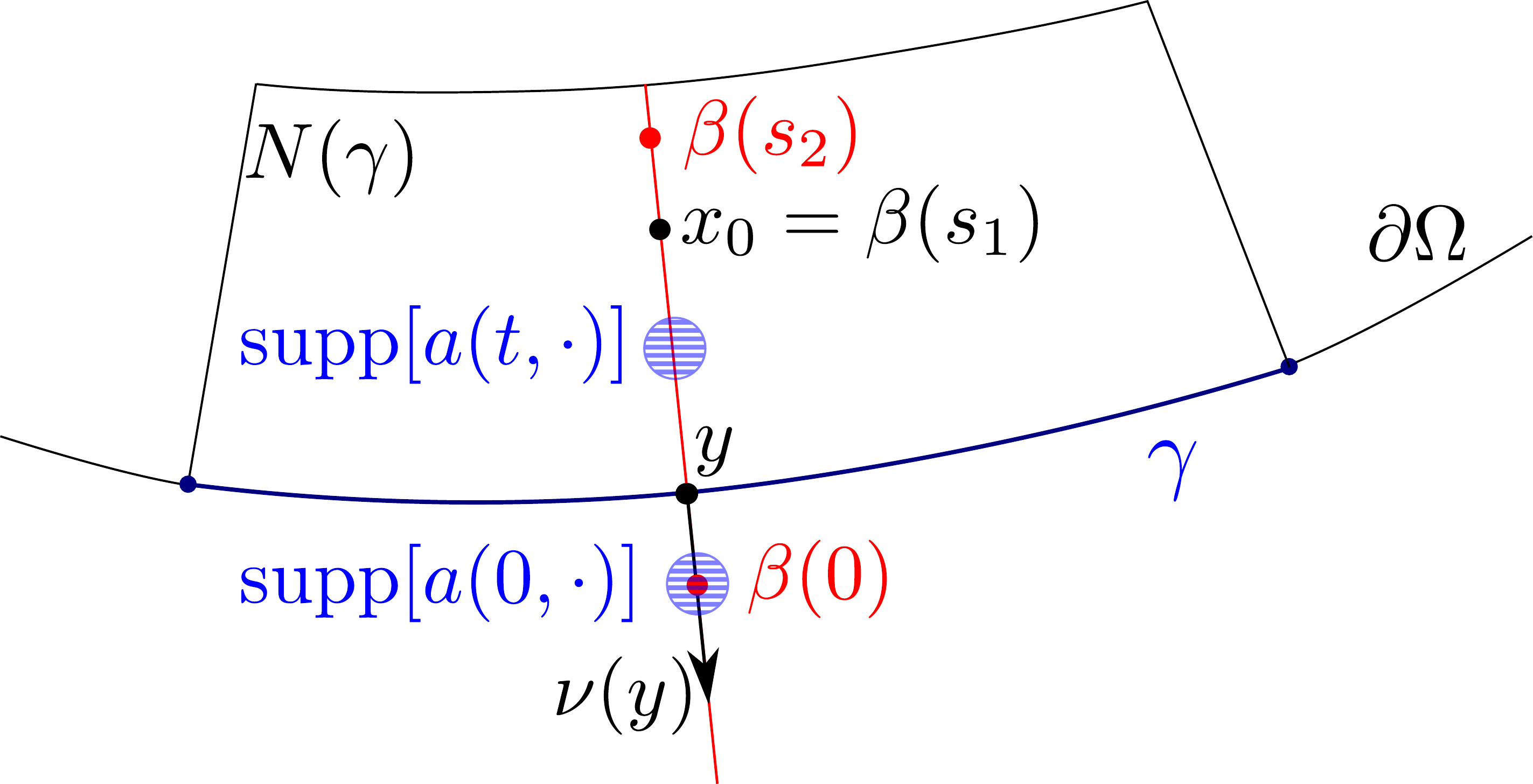}
\caption{\label{FigureSupport} Support of the geometric optics solution}
\end{figure}

To simplify the notation, we suppose that $\chi$ is real valued, and write  $\omega = -\nu(y)$.
Then we consider
$$v_j(t,x)=a(t,x)e^{i\sigma\left(t-x\cdot\omega\right)}+R_{j,\sigma}(t,x),\quad (t,x)\in(0,T)\times\Omega,\ \sigma>1,\ j=1,2,$$
where $R_{j,\sigma}$ solves
$$\left\{\begin{array}{ll}\partial_t^2R_j-\Delta_x R_j+q_j(x)R_j=-e^{i\sigma\left(t-x\cdot\omega\right)}(\p_t^2-\Delta_x+q_j)a,\quad &\textrm{in}\ (0,T)\times\Omega,
\\ R_j=0,\quad &\textrm{on}\ (0,T)\times\p\Omega,
\\  R_j(0,\cdot)=0,\quad \partial_tR_j(0,\cdot)=0,\quad &\textrm{in}\ \Omega.
\end{array}\right.$$
It follows that $\partial_t^2v_j-\Delta_x v_j+q_j(x)v_j=0$ in $(0,T)\times\Omega$ and, analogously to \cite[Lemma 2.2]{KO} one can check that
\bel{t5e}
\norm{R_{j,\sigma}}_{L^2((0,T)\times\Omega)} \to 0,
\quad \sigma \to +\infty.\ee
Moreover, (\ref{suppa}) implies that 
\bel{t5f}\p_t^kv_j(0,x)=\p_t^k\left(a(t,x)e^{i\sigma\left(t-x\cdot\omega\right)}\right)_{|t=0} = 0,\quad x \in \Omega,\ k=0,1.
\ee
Define $g(t,x) = \I_{(0,s_2)}(t) a(t,x)e^{i\sigma\left(t-x\cdot\omega\right)}$, $t \in [0,T]$, $x \in \p \Omega$. 
Then (\ref{suppa}) implies that 
$\supp(g) \subset (0,T) \times \gamma$.

We have $u_j^g = v_j$ on $(0,s_2) \times \Omega$. Up to a reduction of $\delta$ we can assume that $a(t,\cdot)_{|\partial \Omega} = 0$ for $t \geq s_1$, thus $u_j^g = v_j$ on $(0,T) \times \Omega$. Therefore Lemma \ref{l_ptprod} implies that 
for all $f\in \mathcal C^\infty_0((0,T]\times \gamma)$, $\psi\in\mathcal C^\infty_0(-\delta,\delta)$, $t\in[0,T-\delta]$, $s \in \R$, and $x \in N(\gamma)$,
$$u_1^{f}(t+s,x)\overline{v_1(s+s_1,x)}e^{i\sigma(s+s_1-x\cdot\omega)}\psi(s)
=u_2^{f}(t+s,x)\overline{v_2(s+s_1,x)}e^{i\sigma(s+s_1-x\cdot\omega)}\psi(s).$$
Integrating both sides of this expression
and sending $\sigma \to +\infty$, 
we get
\bel{t5k}\int_\R \int_{N(\gamma)} u_1^f(t+s,x)a(s+s_1,x)\psi(s)dxds
=\int_\R \int_{N(\gamma)} u_2^f(t+s,x)a(s+s_1,x)\psi(s)dxds.\ee
After the change of variables $z = x - \beta(s+s_1)$,
\begin{align}\notag 
\int_\R\int_{\R^n} u_1^f(t+s,z+\beta(s+s_1))\chi(z)\psi(s)dzds
=\int_\R\int_{\R^n} u_2^f(t+s,z+\beta(s+s_1))\chi(z)\psi(s)dzds.
\end{align}
As $\chi$ and $\psi$ are arbitrary cutoff functions with small supports,
it holds that $u_1^f(t,x) = u_2^f(t,x)$ for $t \in [0,T-\delta]$ and $x$ near $x_0$. As $\delta > 0$ can be taken arbitrarily small and $x_0 \in N(\gamma)$ can be chosen arbitrarily, the claim follows.

\end{proof}

\subsection{Proof of Theorem \ref{t5}}
Choose $\tau > 0$ and open and non-empty $\gamma' \subset \gamma$ satisfying (\ref{geo}).
Lemmas \ref{l4}, \ref{l_pt} and the finite speed of propagation 
imply that for any $f\in \mathcal C^\infty_0((0,\tau]\times \gamma')$,
\begin{equation*}
\left\langle \phi_{1,k},u_2^{f}(\tau,\cdot)\right\rangle_{L^2(\Omega(\gamma',\tau))}
=\left\langle \phi_{1,k},u_1^{f}(\tau,\cdot)\right\rangle_{L^2(\Omega)}=\left\langle \phi_{2,k},u_2^{f}(\tau,\cdot)\right\rangle_{L^2(\Omega)}=\left\langle \phi_{2,k},u_2^{f}(\tau,\cdot)\right\rangle_{L^2(\Omega(\gamma',\tau))}.
\end{equation*}
This together with Corollary \ref{c1} implies that 
$\phi_{1,k}=\phi_{2,k}$ in $\Omega(\gamma',\tau)$.
Thus in $\Omega(\gamma',\tau)$ it holds that 
$$0=(A_{q_1}-\lambda_k) \phi_{1,k} - (A_{q_2} -\lambda_k)\phi_{2,k}
=(q_1-q_2)\phi_{2,k},\quad k\in\mathbb N^*.$$
Integrating this expression on $\Omega(\gamma',\tau)$ we get
$\left\langle (q_1-q_2)\mathds{1}_{\Omega(\gamma',\tau)},\phi_{2,k} \right\rangle_{L^2(\Omega)}=0$, $k\in\mathbb N^*$.
This proves \eqref{t5a}.\qed

\section{Global recovery}
The goal of this section is to get global recovery of the potential from the local determination.  More precisely, we will complete the proof of Theorem \ref{t1}.
We start by fixing the notation. From now on we consider $\tau\in(0,T)$ and $\gamma' \subset \gamma$  such that condition \eqref{t5a} is fulfilled and  we assume that $T>\diam(\Omega)+2\tau$. For any open  connected set $B$  of $\overline{\Omega}$  we define the operator
$$K_{j,B} h:={u_j^h}_{|[0,2T+\tau]\times B},\quad h\in\mathcal C^\infty_0((0,2T+\tau]\times\gamma').$$
Moreover, for $B\subset\Omega$, we define
$$L_{j,B}F:={v_{j,F}}_{| [0,2T]\times B},\quad F\in\mathcal C^\infty_0((\tau,T)\times B)$$
with $v_{j,F}$ solving
\begin{equation}\label{eq4}\left\{\begin{array}{ll}\partial_t^2v_j-\Delta_x v_j+q_j(x)v_j=F,\quad &\textrm{in}\ (0,\infty)\times\Omega,\\  v_j(0,\cdot)=0,\quad \partial_tv_j(0,\cdot)=0,\quad &\textrm{in}\ \Omega,\\ v_j=0,\quad &\textrm{on}\ (0,\infty)\times\p\Omega.\end{array}\right.\end{equation}
We write also 
$$B(x,r):=\{y\in\R^n:\ |y-x|<r\}, \quad x\in\R^n,\ r>0.$$
\begin{lem}\label{l8} 
 Fix $x\in \Omega(\gamma',\tau)$ and consider for $\epsilon>0$ the set $B=B(x,\epsilon)\cap \Omega(\gamma',\tau)$. Then,   we have
\bel{l8a}K_{1,B}=K_{2,B}.\ee
Moreover, if $\overline{B}\subset \Omega\cap \Omega(\gamma',\tau)$ we have
\bel{l8b}L_{1,B}=L_{2,B}.\ee
\end{lem}
\begin{proof}
The equation (\ref{l8a}) follows immediately from Lemma \ref{l_pt}. 
Let $h \in \mathcal C^\infty_0((0,2T+\tau)\times\gamma')$
and set $\tilde h(t,x) = h(2T+\tau-t,x)$.
Then integrating by parts, we find
$$\begin{aligned}\int_0^{2T+\tau}\int_{B}u_j^{\tilde{h}}(2T+\tau-t,x)\overline{F(t,x)}dxdt&=\int_0^{2T+\tau}\int_\Omega u_j^{\tilde{h}}(2T+\tau-t,x)\overline{F(t,x)}dxdt\\
\ &=\int_0^{2T+\tau}\int_\Omega u_j^{\tilde{h}}(2T+\tau-t,x)\overline{(\p_t^2-\Delta_x+q_j)v_j(t,x)}dxdt\\
\ &=-\int_0^{2T+\tau}\int_{\gamma'} \tilde{h}(2T+\tau-t,x)\overline{\p_\nu v_j}dxdt.\end{aligned}$$
Then \eqref{l8a} implies
$$\int_0^{2T+\tau}\int_{\gamma'} h\overline{\p_\nu v_1}dxdt=\int_0^{2T+\tau}\int_{\gamma'} h\overline{\p_\nu v_2}dxdt.$$
As $h \in \mathcal C^\infty_0((0,2T+\tau)\times\gamma')$ is arbitrary we deduce that  $\p_\nu v_1=\p_\nu v_2$ on $(0,2T+\tau)\times\gamma'$. 
Thus, fixing $v=v_1-v_2$  and using \eqref{t5a}, we deduce that   $\p_t^2v-\Delta_x v+q_1v=0$ on $(0,2T+\tau)\times\Omega(\gamma',\tau)$, $v_{|(0,2T+\tau)\times\gamma'}=\p_\nu v_{|(0,2T+\tau)\times\gamma'}=0$. In particular, for any $t\in(\tau,2T)$, we have
$\p_t^2v-\Delta_xv +q_1v=0$ on $(t-\tau,t+\tau)\times\Omega(\gamma',\tau)$, $v_{|[t-\tau,t+\tau]\times\gamma'}=\p_\nu v_{|[t-\tau,t+\tau]\times\gamma'}=0$. Thus, the unique continuation property of Theorem \ref{t3} implies that, for any $t\in(\tau,2T)$, we have $v_1(t,\cdot)=v_2(t,\cdot)$ on $\Omega(\gamma',\tau)\supset B$. On the other hand, using the fact that $$\textrm{supp}(F)\subset (\tau,T)\times B\subset(\tau,T)\times\Omega,$$ one can check that the restriction of $v_j$ to $(0,\tau)\times\Omega$ solves the problem
$$\left\{\begin{array}{ll}\partial_t^2v_j-\Delta_x v_j+q_j(x)v_j=0,\quad &\textrm{in}\ (0,\tau)\times\Omega,\\  v_j(0,\cdot)=0,\quad \partial_tv_j(0,\cdot)=0,\quad &\textrm{in}\ \Omega,\\ v_j=0,\quad &\textrm{on}\ (0,\tau)\times\p\Omega.\end{array}\right.$$
which implies that $v_1(t,x)=0=v_2(t,x)$, $(t,x)\in(0,\tau)\times\Omega$. Combining these two identities we deduce  \eqref{l8b}.

\end{proof}

We extend the notion of domain of influence for any $r > 0$ and any open set $B \subset \Omega$ by setting
\[
\Omega(B,r) = \left\{ x \in \overline{\Omega} : \dist(x,B) \leq r \right\}.
\]
From now on, we fix $\epsilon_0\in(0,\tau/7)$, $x_0\in\Omega(\gamma',\tau-3\epsilon_0)$, dist$(x_0,\Gamma)>3\epsilon_0$, $B=B(x_0,\epsilon_0)$. Note that $\Omega(B,\epsilon_0)\subset int(\Omega(\gamma',\tau))$. 
In the remaining of this text we will prove that \eqref{l8b} implies that $q_1=q_2$ in $\Omega\setminus B$. Note first that according to the finite speed of propagation we have 
\bel{fsp}\textrm{supp}(v_{j,F}(T,\cdot))\subset \Omega(B,r),\quad F\in \mathcal C^\infty_0((T-r,T)\times B).\ee
This is the analogue of Theorem \ref{t2} for solutions of (\ref{eq4}).

We will also need to use global unique continuation in the domain $\Omega \setminus B$.
In the appendix, we discuss unique continuation only under assumptions that allow us to avoid certain arguments of geometric nature. 
For this purpose let $\Omega' \subset \R^n$ be a domain with smooth boundary. We fix $S$ an open subset of $\partial \Omega'$ and for every $x\in\Omega'$, we consider the set $Z_x(S)=\{y\in\overline{S}:\ |x-y|=\textrm{dist}(x,S)\}$. Then, we introduce the following condition on $S$:

\medskip
\noindent
$\ \ $(H) Let $x\in \Omega'$. If $y\in Z_x(S)\cap S$ then $[x,y]:=\{tx+(1-t)y:\ t\in[0,1]\}\subset(\Omega'\cup S)$. Furthermore, if $y\in Z_x(S)\setminus S$ then for every neighborhood $V$ of $y$ in $\overline{S}$ there exists $z\in V\cap S$ such that $[x,z]\subset(\Omega'\cup S)$.

\medskip
As $\Omega$ is convex, this condition holds for any $S \subset \p \Omega$.
Now,  let us recall (as illustrated in Fig.~\ref{FigureConvexeBoule}) that for any $x\in \Omega\setminus B$, 
$y=x_0+\epsilon_0 \frac{x-x_0}{|x-x_0|}$ is the unique element of $\overline{B}$ satisfying dist$(x, B)=|y-x|$. Moreover, since $\Omega$ is convex we have $[x,y]\subset\Omega$ and $[x,y]$ can not meet $B$ since $y$ is the unique element of $\p B$ satisfying dist$(x,\p B)=|y-x|$. Therefore, we have $[x,y]\subset (\Omega\setminus B)\cup\p B$ and $\p B$ satisfies condition (H).
Theorem \ref{t3} with $\Omega$ replaced by $\Omega \setminus B$
and $S=\p B$ implies the following analogue of Corollary \ref{c1}.

\medskip
\begin{figure}[h]
\includegraphics[width=.4\textwidth]{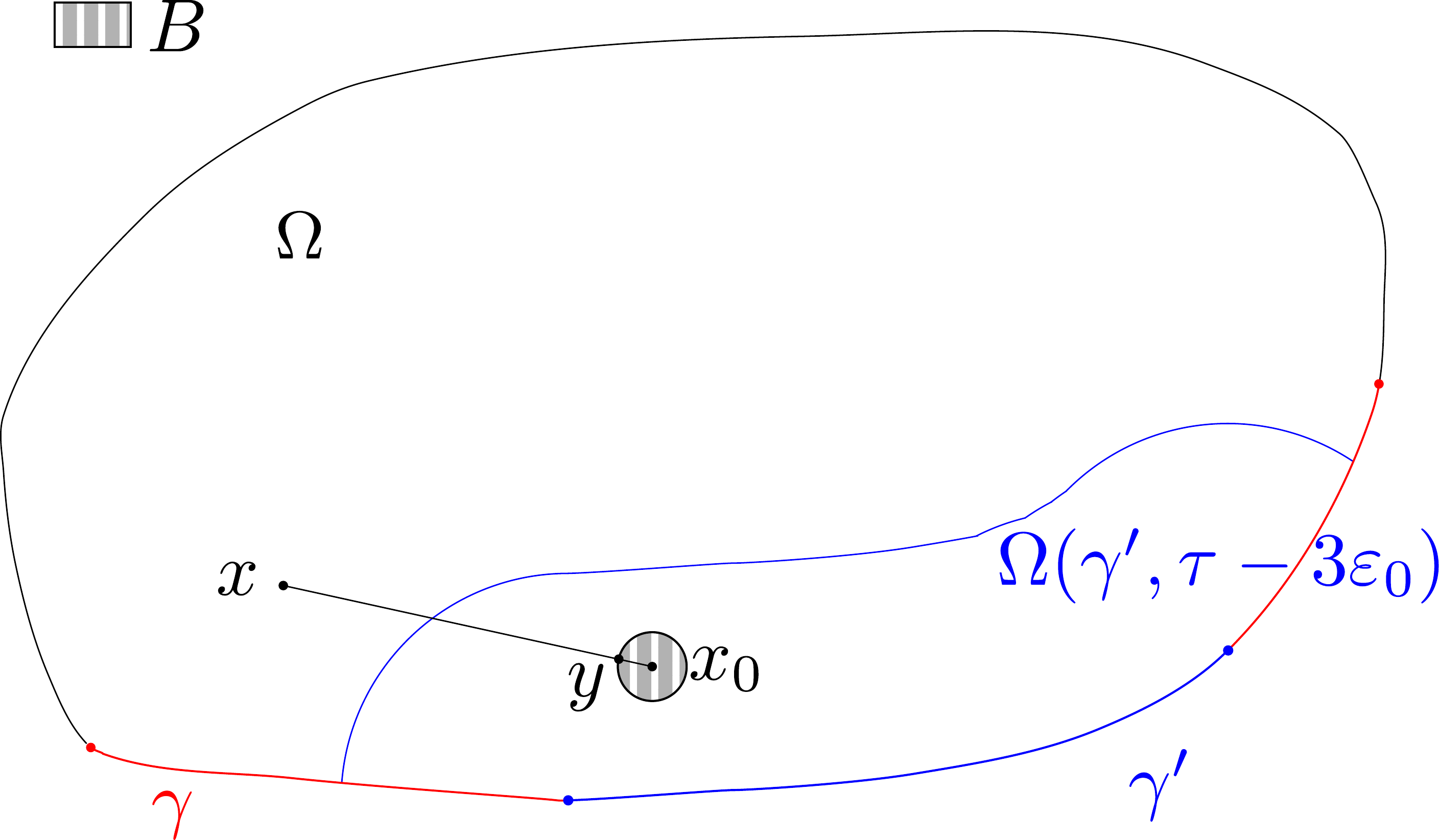}
\caption{\label{FigureConvexeBoule}}
\end{figure}

\begin{lem}\label{ll1} 
For $j=1,2$, the set 
\bel{ll1a}\{v_{j,F}(T,\cdot):\ F\in \mathcal C^\infty_0((T-r,T)\times B)\}\ee
is dense in $L^2(\Omega(B,r))$.
\end{lem}

For convenience of the reader, we have included a proof of this lemma in the appendix. 
The proof is analogous with the proof of Corollary \ref{c1}.
Let us also show that  the norm of solution of \eqref{eq4} with $F\in\mathcal C^\infty_0((\tau ,T)\times B)$ is determined by condition \bel{l9a}L_{1,\Omega(B,\epsilon_0)}=L_{2,\Omega(B,\epsilon_0)},\ee which follows from \eqref{l8b} and the fact that $\Omega(B,\epsilon_0)\subset int(\Omega(\gamma',\tau))$.
We have the following analogue of (\ref{id2})

\begin{lem}\label{l10} 
Condition \eqref{l9a} implies that, for all $t,s\in[0 ,T]$ and all $F_1,F_2\in \mathcal C^\infty_0((\tau,T)\times B)$, we have
\bel{l10a}\left\langle v_{1,F_1}(t,\cdot),v_{1,F_2}(s,\cdot)\right\rangle_{L^2(\Omega)}=\left\langle v_{2,F_1}(t,\cdot),v_{2,F_2}(s,\cdot)\right\rangle_{L^2(\Omega)}.\ee

\end{lem}
\begin{proof}
Consider for $j=1,2$ and for all $t,s\in(0,2T]$ the function $w_j(t,s)=\left\langle v_{j,F_1}(t,\cdot),v_{j,F_2}(s,\cdot)\right\rangle_{L^2(\Omega)}$.
Integrating by parts, for all $t,s\in(0 ,2T)$,  we find
$$\begin{aligned} &(\p_t^2-\p_s^2)w_j(t,s)=\left\langle \p_t^2v_{j,F_1}(t,\cdot),v_{j,F_2}(s,\cdot)\right\rangle_{L^2(\Omega)}-\left\langle v_{j,F_1}(t,\cdot),\p_s^2v_{j,F_2}(s,\cdot)\right\rangle_{L^2(\Omega)}\\
\ &=\left\langle (\Delta_x-q_j)v_{j,F_1}(t,\cdot)+F_1(t,\cdot),v_{j,F_2}(s,\cdot)\right\rangle_{L^2(\Omega)}-\left\langle v_{j,F_1}(t,\cdot),(\Delta_x-q_j)v_{j,F_2}(s,\cdot)+F_2(s,\cdot)\right\rangle_{L^2(\Omega)}\\
\ &=\left\langle F_1(t,\cdot),L_{j,\Omega(B,\epsilon)}F_2(s,\cdot)\right\rangle_{L^2( B)}-\left\langle L_{j,\Omega(B,\epsilon)}F_1(t,\cdot),F_2(s,\cdot)\right\rangle_{L^2( B)}.\end{aligned}$$
Then, applying \eqref{l9a}, we deduce that, for all $t,s\in(0,2T)$, we have
\bel{bb}(\p_t^2-\p_s^2)w_1(t,s)=\left\langle F_1(t,\cdot),L_{1,\Omega(B,\epsilon)}F_2(s,\cdot)\right\rangle_{L^2( B)}-\left\langle L_{1,\Omega(B,\epsilon)}F_1(t,\cdot),F_2(s,\cdot)\right\rangle_{L^2( B)}=(\p_t^2-\p_s^2)w_2(t,s).\ee
Moreover, for $j=1,2$, we have $w_j(0,s)=\p_tw_j(0,s)=w_j(t,0)=\p_sw_j(t,0)=0$. Thus, applying \eqref{bb}, we deduce that $w=w_1-w_2$ solves the system associated with the $1+1$ wave equation
$$\left\{\begin{array}{ll}\partial_t^2w-\p_s^2w=0,\quad &\textrm{in}\ (0,2T)\times(0,2T),\\  w(0,\cdot)=0,\quad \partial_tw(0,\cdot)=0,\quad &\textrm{in}\ (0,2T),\\  w(\cdot,0)=0,\quad \partial_sw(\cdot,0)=0,\quad &\textrm{in}\ (0,2T).\end{array}\right.$$

From unique continuation we can deduce that 
\[
w(T,s) = \partial_t w(T,s) = 0, \quad \forall s \in \Omega(\{0\},T) = (0,T).
\]

Therefore, $w$ solves 
$$\left\{\begin{array}{ll}
\partial_s^2w-\p_t^2w=0,\quad &\textrm{in}\ (0,T)\times(0,T),\\  
w(0,s) = w(T,s) = 0, \quad\quad &s \in (0,T),\\
w(t,0) = \partial_s w(t,0) = 0, \quad\quad &t \in (0,T).
\end{array}\right.$$

Then, the uniqueness of this initial boundary value problem implies that $w_{|(0,T)\times(0,T)}=0$ which, according to the continuity of $w$ with respect to $(t,s)\in[0,T]^2$,  implies that for all $t,s\in[0,T]$  we have $w_1(t,s)=w_2(t,s)$. 
This proves \eqref{l10a}.

\end{proof}

\begin{proof}[Proof of Theorem \ref{t1}]

In a similar way as in the previous section, for any $x\in int(\Omega(B,t)\setminus B)$, $t\in(0,T]$, we consider $y\in\p B$ the unique element of $B$ such that  dist$(x,B)=|x-y|=s\in(0,\tau)$. One can check that there exist $z_1,\ldots,z_n\in B$ such that $(x-z_1,x-z_2,\ldots,x-z_n)$ is a basis of $\R^n$. Moreover,  we can choose  $z_1,\ldots,z_n\in B$ such that for
\bel{Ax11}B_{x,\epsilon}':=\bigcap_{j=1}^nB(z_j,|x-z_j|+\epsilon),\quad A_{x,\epsilon}':=B_{x,\epsilon}'\setminus \Omega(B,s-\epsilon),\ee
the family $(A_{x,\epsilon}')_{\epsilon>0}$ is of bounded eccentricity and 
\bel{Ax22}\lim_{\epsilon\to0}A_{x,\epsilon}'=\{x\}.\ee
Combining the density of the set \eqref{l10a} with the arguments used in Lemma \ref{l_ptprod}, we obtain 
$$
\abs{A_{x,\epsilon}'}^{-1}\left\langle \mathds{1}_{A_{x,\epsilon}'}v_{1,F}(\tau,\cdot),v_{1,G}(\tau,\cdot)\right\rangle_{L^2(\Omega)}.
= \abs{A_{x,\epsilon}'}^{-1}\left\langle \mathds{1}_{A_{x,\epsilon}'}v_{2,F}(\tau,\cdot),v_{2,G}(\tau,\cdot)\right\rangle_{L^2(\Omega)}
$$
for all $F,G \in \mathcal C_0^\infty((\tau, T) \times B)$.
After taking the limit $\epsilon \to 0$, we obtain 
\bel{t1c}v_{1,F}(t,x)\overline{v_{1,G}(t,x)}=v_{2,F}(t,x)\overline{v_{2,G}(t,x)},\quad t\in(\tau,T],\ \ x\in  \Omega,\ F,G\in \mathcal C^\infty_0((\tau,T)\times B).\ee
From now on our goal will be to use this identity to conclude. For this purpose, in a similar way to Theorem \ref{t5} we will use special solutions that we will introduce next in order to recover
$v_{1,F}(T,\cdot)$ for any $F\in \mathcal C^\infty_0((\tau,T)\times B)$. Then we will complete the proof. Let us first fix $x_1\in\Omega\setminus B$ and consider $s_1=$dist$(x_1,B)$. Since $T-\tau> \diam(\Omega)$ we know that $s_1+\epsilon_0=|x_1-x_0|\in (0,T-\tau)$. Now consider $$\delta_1\in(0,s_1/2)\cap(0,\epsilon_0/2)\cap(0,\textrm{dist}(x_1,\p\Omega)/2)(0,(T-s_1-\epsilon_0-\tau)/2).$$ We fix also
$$\omega=\frac{x_1-x_0}{|x_1-x_0|},\quad \delta_2:=\inf_{t\in[0,s_1+\epsilon_0+\delta_1]}\textrm{dist}\left(x_0+t\omega,\p\Omega\right)$$
and we consider  $\delta=\frac{\min(\delta_1,\delta_2)}{4}$. Note that according to the definition of $\delta_1$ we have $\delta_2>0$. Now let $\chi\in\mathcal C^\infty_0(B(0,\delta))$, $\delta'\in(0,\delta)$ and define
$$a(t,x):= \chi\left(x-x_0-(t-T+s_1+\epsilon_0+\delta')\omega\right).$$
Then, we consider,
$$u_j(t,x)=a(t,x)e^{i\sigma\left(t-x\cdot\omega\right)}+R_{j,\sigma}(t,x),\quad (t,x)\in(0,\tau)\times\Omega,\ j=1,2,\ \sigma>1,$$
where $R_{j,\sigma}$ solves
\bel{t1d}\left\{\begin{array}{ll}\partial_t^2R_j-\Delta_x R_j+q_j(x)R_j=-e^{i\sigma\left(t-x\cdot\omega\right)}(\p_t^2-\Delta_x+q_j)a,\quad &\textrm{in}\ (0,T)\times\Omega,\\  R_j(T-s_1-\epsilon_0,\cdot)=0,\quad \partial_tR_j(T-s_1-\epsilon_0,\cdot)=0,\quad &\textrm{in}\ \Omega,\\ R_j=0,\quad &\textrm{on}\ (0,T)\times\p\Omega.\end{array}\right.\ee
It follows that $\partial_t^2u_j-\Delta_x u_j+q_j(x)u_j=0$ in $(0,T)\times\Omega$ and one can check that
\bel{t1e}\norm{R_{j,\sigma}}_{L^2((0,\tau)\times\Omega)}\leq C\sigma^{-1},\ee
with $C$ independent of $\sigma$. Then, we consider $\beta\in\mathcal C^\infty(\R;[0,1])$ satisfying $\beta(t)=0$ for $t\in(-\infty, T-s_1-\epsilon_0]$ and $\beta(t)=1$ for $t\in[ T-s_1-\epsilon_0+\delta,+\infty)$. Then, we introduce $$w_j(t,x):=\beta(t)u_j(t,x).$$
It is clear that
\bel{t1f} w_j(t,x)=0,\quad (t,x)\in [0,T-s_1-\epsilon_0]\times\Omega.\ee
Moreover, in  view of \eqref{t1d}, we know that
$$w_j(t,x)=a(t,x)=\overline{\chi}\left(x-x_0-(t-T+s_1+\epsilon_0+\delta')\omega\right)e^{i\sigma\left(t-x\cdot\omega\right)},\quad (t,x)\in(T-s_1-\epsilon_0,T)\times \p\Omega.$$
Using the fact that
\bel{tt}\textrm{dist}(x_0+(t-T+s_1+\epsilon_0+\delta')\omega,\p\Omega)\geq \delta_2\geq4\delta,\quad t\in [T-s_1-\epsilon_0,T]\ee
we deduce that
$$a(t,x)=0,\quad (t,x)\in [T-s_1-\epsilon_0,T]\times\p\Omega.$$
Combining this with, \eqref{t1f} we deduce that $$w_j(t,x)=0,\quad (t,x)\in(0,T)\times\p\Omega.$$
In the same way, \eqref{t1f} implies that $$w_j(0,x)=\p_tw_j(0,x)=0,\quad x\in\Omega$$
and fixing $G_j=2\beta'(t)\p_tu_j+\beta''(t)u_j$, we deduce that $w_j=v_{j,G_j}$. Now let us show that $G_1=G_2=G$ with supp$(G)\subset (\tau,T)\times B$. For this purpose note first that
\bel{t1g}G_1-G_2=2\beta'(t)\p_t(R_{1,\sigma}-R_{2,\sigma})+\beta''(t)(R_{1,\sigma}-R_{2,\sigma}).\ee
On the other hand, we have
$$|(t-T+s_1+\epsilon_0+\delta')\omega+y|\leq 2\delta+|t-T+s_1+\epsilon_0|\leq 3\delta,\quad y\in B(0,\delta),\ t\in[T-s_1-\epsilon_0,T-s_1-\epsilon_0+\delta]$$
and we deduce that  $$\textrm{supp}(a)\cap[T-s_1-\epsilon_0,T-s_1-\epsilon_0+\delta]\times\R^n\subset[T-s_1-\epsilon_0,T-s-\epsilon_0+\delta]\times B(x_0,3\delta).$$ Thus, condition \eqref{t5a} implies that, for all $(t,x)\in (T-s_1-\epsilon_0,T-s_1-\epsilon_0+\delta)\times\Omega$, we get
$$(\p_t^2-\Delta_x+q_j)R_{j,\sigma}(t,x)=-e^{i\sigma\left(t-x\cdot\omega\right)}(\p_t^2-\Delta_x+q_j)a(t,x)=-e^{i\sigma\left(t-x\cdot\omega\right)}(\p_t^2-\Delta_x+q_1)a(t,x).$$
Moreover, condition \eqref{fsp} implies \bel{tt1d}\textrm{supp}(R_{j,\sigma})\cap[T-s_1-\epsilon_0,T-s_1-\epsilon_0+\delta]\times\Omega\subset [T-s_1-\epsilon_0,T-s_1-\epsilon_0+\delta]\times \Omega(B(x_0,3\delta),\delta)\subset (\tau,T)\times B\ee
and, in virtue of \eqref{t5a}, we deduce that
$$(\p_t^2-\Delta_x+q_j)R_{j,\sigma}(t,x)=(\p_t^2-\Delta_x+q_1)R_{j,\sigma}(t,x),\quad (t,x)\in(T-s_1-\epsilon_0,T-s_1-\epsilon_0+\delta)\times\Omega,\ j=1,2.$$
Combining this with \eqref{t1d}, we deduce that the restriction of $R_{j,\sigma}$, $j=1,2$, to $(T-s-\epsilon_0,T-s-\epsilon_0+\delta)\times\Omega$ solves 
$$\left\{\begin{array}{ll}\partial_t^2R-\Delta_x R+q_1(x)R=-e^{i\sigma\left(t-x\cdot\omega\right)}(\p_t^2-\Delta_x+q_1)a,\quad &\textrm{in}\ (T-s-\epsilon_0,T-s-\epsilon_0+\delta)\times\Omega,\\  R(T-s_1-\epsilon_0,\cdot)=0,\quad \partial_tR(T-s_1-\epsilon_0,\cdot)=0,\quad &\textrm{in}\ \Omega,\\ R=0,\quad &\textrm{on}\ (T-s-\epsilon_0,T-s-\epsilon_0+\delta)\times\p\Omega.\end{array}\right.$$
The uniqueness of solutions for this IBVP implies that 
$$R_{1,\sigma}(t,x)=R_{2,\sigma}(t,x),\quad (t,x)\in(T-s_1-\epsilon_0,T-s_1-\epsilon_0+\delta)\times\Omega.$$
Combining this with \eqref{t1g}, we deduce  that $G_1=G_2=G$ and \eqref{tt1d} implies that supp$(G)\subset (\tau,T)\times B$. 
Applying \eqref{t1c},  one can check  that, for all $F\in \mathcal C^\infty_0((\tau,T)\times B)$ and $\psi\in\mathcal C^\infty_0((\tau,T))$, we have
$$v_{1,F}(t,x)\overline{w_1(t,x)}e^{i\sigma(t+x\cdot\omega)}\psi(t)=v_{2,F}(t,x)\overline{w_2(t,x)}e^{i\sigma(t-x\cdot\omega)}\psi(t),\quad t\in[\tau,T],\ \ x\in  \Omega.$$
Integrating both sides of this expression, we get
$$\begin{array}{l}\int_\tau^T\int_\Omega v_{1,F}\beta(t)\overline{a}\psi(t)dxdt+\int_\tau^T\int_\Omega v_{1,F}\beta(t)\overline{R_{1,\sigma}}e^{i\sigma(t-x\cdot\omega)}\psi(t)dxdt\\
=\int_\tau^T\int_\Omega v_{2,F}\overline{a}\beta(t)\psi(t)dxdt+\int_\tau^T\int_\Omega v_{2,F}\overline{R_{2,\sigma}}e^{i\sigma(t-x\cdot\omega)}\beta(t)\psi(t)dxdt.\end{array}$$
Then, in view of \eqref{t1e}, sending $\sigma\to+\infty$ and using the fact that $\beta=0$ on $[0,\tau]$, we get
\bel{t1k}\int_0^T\int_\Omega v_{1,F}(t,x)\overline{a(t,x)}\psi(t)\beta(t)dxdt=\int_0^T\int_\Omega v_{2,F}(t,x)\overline{a(t,x)}\psi(t)\beta(t)dxdt.\ee
Consider $\theta\in\mathcal C^\infty_0\left(\left(-\delta',\delta'\right)\right)$ and fix $\psi(t):=\theta\left(t-T+\delta'\right)$. 
Note that supp$(\psi)\subset \left(T-2\delta',T\right)$ and \eqref{tt} implies that,
 for all $t\in(0,T)$, supp$(\beta(t)a(t,\cdot)\psi(t))\subset\Omega$. Thus, \eqref{t1k} becomes
$$\int_\R\int_{\R^n} \beta(t)v_{1,F}(t,x)\overline{a(t,x)}\psi(t)dxdt=\int_\R\int_{\R^n} \beta(t)v_{2,F}(t,x)\overline{a(t,x)}\psi(t)dxdt.$$
Making the substitution $z=x-x_0-(t-T+s_1+\epsilon_0+\delta')\omega$, we obtain
$$\begin{aligned}\int_\R\int_{\R^n} \beta(t)v_{1,F}\left(t,z+x_0+(t-T+s_1+\epsilon_0+\delta')\omega\right)\chi(z)\psi(t)dzdt\\
=\int_\R\int_{\R^n} \beta(t)v_{2,F}\left(t,z+x_0+(t-T+s_1+\epsilon_0+\delta')\omega\right)\chi(z)\psi(t)dzdt\end{aligned}$$
and making the substitution $s=t-T+\delta' $, we find 
$$\begin{aligned}\int_\R\int_{\R^n} \beta(s+T-\delta') v_{1,F}\left(s+T-\delta,z+x_0+\left(s+s_1+\epsilon_0\right)\omega\right)\chi(z)\theta(s)dzdt\\
=\int_\R\int_{\R^n}\beta(s+T-\delta') v_{2,F}\left(s+T-\delta,z+x_0+\left(s+s_1+\epsilon_0\right)\omega\right)\chi(z)\theta(s)dzdt.\end{aligned}$$
Allowing $\chi\in\mathcal C^\infty_0(B(0,\delta))$, $\theta\in\mathcal C^\infty_0((-\delta',\delta'))$ to be arbitrary, we deduce that for almost every $z\in B(0,\delta)$ we have
$$v_{1,F}\left(s+T-\delta',z+x_0+\left(s+s_1+\epsilon_0\right)\omega\right)=v_{2,F}\left(s+T-\delta',z+x_0+\left(s+s_1+\epsilon_0\right)\omega\right),\quad s\in(-\delta',\delta')$$
which, by fixing $s=0$, implies that
$$v_{1,F}\left(T-\delta',z+x_0+(\epsilon_0+s_1)\omega\right)=v_{2,F}\left(T-\delta',z+x_0+(\epsilon_0+s_1)\omega\right).$$
Then, using the fact that $x_1=x_0+(\epsilon_0+s_1)\omega$, we deduce that 
$$v_{1,F}\left(T-\delta',z+x_1\right)=v_{2,F}\left(T-\delta',z+x_1\right),\quad\textrm{a.e. }z\in B(0,\delta).$$
Sending $\delta'\to0$,  we prove that
$$v_{1,F}\left(T,z\right)=v_{2,F}\left(T,z\right),\quad\textrm{a.e. }z\in B(x_1,\delta).$$
Now allowing $x_1\in \Omega\setminus B$ to be arbitrary we deduce that 
$$v_{1,F}(T,x)=v_{2,F}(T,x),\quad \ x\in \Omega\setminus B.$$
Then, \eqref{l9a} implies
$$v_{1,F}(T,x)=v_{2,F}(T,x),\quad  x\in \Omega$$
and we get
\bel{t1m}v_{2,F}(T,\cdot)=v_{1,F}(T,\cdot),\quad F\in \mathcal C^\infty_0((\tau,T)\times B).\ee
Using this identity we will complete the proof of the theorem. For this purpose, note first that repeating the arguments of Lemma \ref{l4} one can check that 
$$\left\langle v_{1,F}(T,\cdot),\phi_{1,k}\right\rangle_{L^2(\Omega)}=\left\langle v_{2,F}(T,\cdot),\phi_{2,k}\right\rangle_{L^2(\Omega)},\quad F\in \mathcal C^\infty_0((\tau,T)\times B),\ k\in\mathbb N^*.$$
Then, \eqref{t1m} implies that
$$\left\langle v_{1,F}(T,\cdot),\phi_{1,k}-\phi_{2,k}\right\rangle_{L^2(\Omega)}=0,\quad F\in \mathcal C^\infty_0((\tau,T)\times B),\ k\in\mathbb N^*$$
and the density results of Lemma \ref{ll1} implies
$$\phi_{1,k}(x)=\phi_{2,k}(x),\quad x\in\Omega,\ k\in\mathbb N^*.$$
Combining this with arguments similar to the end of the proof of Theorem \ref{t5} we deduce that $q_1=q_2$.

\end{proof}

\section{Recovery from the Dirichlet-to-Neumann map}

This section is devoted to proof of Theorem \ref{tt1}. The proof of this result is similar to the one of Theorem \ref{t1} and the boundary spectral data $B(q,\gamma)$ can be replaced by the Dirichlet-to-Neumann map $\Lambda_q$ as far as $T> 2\diam(\Omega)$. The only point that we need to check is the following. 

\begin{lem}\label{l10b} Let $q_1,q_2\in L^\infty(\Omega)$. For $f\in H^1(\Sigma)$ satisfy $f_{|t=0}=0$, supp$(f)\subset [0,T]\times \gamma'$ and, for $j=1,2$, we fix $u_j^f$ be the solution  of \eqref{eq1} with $q=q_j$. 
Condition $\Lambda_{q_1}=\Lambda_{q_2}$ implies that, for all $t,s\in(0,T/2]$ and all $h,f\in \mathcal C^\infty_0((0,T]\times\gamma)$, we have
\bel{l10ab}\left\langle u^h_1(t,\cdot),u^f_1(s,\cdot)\right\rangle_{L^2(\Omega)}=\left\langle u^h_2(t,\cdot),u^f_2(s,\cdot)\right\rangle_{L^2(\Omega)}.\ee
\end{lem}

The proof is similar with that of Lemma \ref{l10},
however, we give it for the convenience of the reader.

\begin{proof} 
Consider for $j=1,2$ and for all $t,s\in[0,T]$ the function $w_j(t,s)=\left\langle u^h_j(t,\cdot),u^f_j(s,\cdot)\right\rangle_{L^2(\Omega)}$.
Then, integrating by parts, we find
$$\begin{aligned} \p_t^2w_j-\p_s^2w_j&=\left\langle \p_t^2u^h_j(t,\cdot),u^f_j(s,\cdot)\right\rangle_{L^2(\Omega)}-\left\langle u^h_j(t,\cdot),\p_s^2u^f_j(s,\cdot)\right\rangle_{L^2(\Omega)}\\
\ &=\left\langle (\Delta_x-q_j)u^h_j(t,\cdot),u^f_j(s,\cdot)\right\rangle_{L^2(\Omega)}-\left\langle u^h_j(t,\cdot),(\Delta_x-q_j)u^f_j(s,\cdot)\right\rangle_{L^2(\Omega)}\\
\ &=\left\langle \Lambda_{q_j}h(t,\cdot),f(s,\cdot)\right\rangle_{L^2(\gamma)}-\left\langle h(t,\cdot),\Lambda_{q_j}f(s,\cdot)\right\rangle_{L^2(\gamma)}.\end{aligned}$$
Moreover, for $j=1,2$, we have $w_j(0,s)=\p_tw_j(0,s)=w_j(t,0)=\p_sw_j(t,0)=0$. Thus, applying \eqref{l9a}, we deduce that $w=w_1-w_2$ solves the system associated with the $1+1$ wave equation
$$\left\{\begin{array}{ll}\partial_t^2w-\p_s^2w=0,\quad &\textrm{in}\ (0,T)\times(0,T),\\  w(0,\cdot)=0,\quad \partial_tw(0,\cdot)=0,\quad &\textrm{in}\ (0,T),\\  w(\cdot,0)=0,\quad \partial_sw(\cdot,0)=0,\quad &\textrm{in}\ (0,T).\end{array}\right.$$

Repeating the argument developed in the proof of Lemma~\ref{l10} leads to $w_{(0,T/2) \times (0,T/2)}=0$.
This proves~\eqref{l10ab}.

\end{proof}

Using \eqref{l10ab} and repeating the arguments used for Theorem \ref{t5}, we can show that
$$ q_1(x)=q_2(x),\quad x\in\Omega(\gamma',\tau)$$
for some $\gamma'\subset\gamma$ and $\tau\in (0,T)$. Fixing $\tau<\frac{T-2\diam(\Omega)}{4}$, we consider the following
\begin{lem}\label{l11} Let $q_1,q_2\in L^\infty(\Omega)$ satisfy \eqref{t5a}. Then the condition $\Lambda_{q_1}=\Lambda_{q_2}$ implies that

\bel{l11a}u_1^f(t,x)=u_2^f(t,x),\quad (t,x)\in(\tau,T-\tau)\times\Omega(\gamma',\tau),\quad f\in\mathcal C^\infty_0((0,T]\times\gamma').\ee
Moreover, we get
\bel{l11b}v_{1,F}(t,x)=v_{2,F}(t,x),\quad (t,x)\in(\tau,T-2\tau)\times B,\quad F\in\mathcal C^\infty_0((2\tau,T-\tau]\times B).\ee
\end{lem}
\begin{proof} Note first that according to \eqref{t5a},  $u=u_1^f-u_2^f$ satisfies $\partial_t^2u-\Delta_x u+q_1u=0$ on $(0,T)\times\Omega(\gamma',\tau)$ and the condition $\Lambda_{q_1}=\Lambda_{q_2}$ implies $u_{|(0,T)\times\gamma'}=\partial_\nu u_{|(0,T)\times\gamma'}=0$. Thus, from the unique continuation property of Theorem \ref{t3} we deduce \eqref{l11a}. In view of \eqref{l11a}, we deduce \eqref{l11b} by mimicking the proof of statement \eqref{l8b} in Lemma \ref{l8} .\end{proof}

Armed with this lemma and  the arguments used for the global recovery in Section 4 we can complete the proof of Theorem \ref{tt1}.

\appendix
\section{Proofs of classical results on wave equations}

We begin by proving the reformulation of the finite speed of progation as stated in Section 2.1.

\begin{proof}[Proof of Theorem \ref{t2}]
Let $x_0\in\Omega\setminus  \Omega(S,\tau)$. Then we have dist$(x_0,S)=d>\tau$.  Fixing $\mathcal O:=\{x\in\R^n:\ \textrm{dist}(x,S)<\frac{d-\tau}{3}\}$,  $\widetilde{D}:=\{(t,x)\in [0,T]\times\overline{\Omega}:\ \textrm{dist}(x,\mathcal O)>t\}$ and applying Lemma \ref{l1}, we obtain $u_{|\widetilde{D}}=0$. For $t=\tau$ and any $x\in\{y\in\Omega:\ |y-x_0|<\frac{d-\tau}{3}\}$, we obtain that $$\textrm{dist}(x,\mathcal O)>\textrm{dist}(x_0,\mathcal O)- \frac{d-\tau}{3} \geq\textrm{dist}(x_0,S)- \frac{2(d-\tau)}{3}
=\tau+ \frac{d-\tau}{3}>\tau,$$
which implies that $(\tau,x)\in \widetilde{D}$. Thus, $u(\tau,\cdot)=0$ on a neighborhood of $x_0$ and $x_0\notin \supp \left[u(\tau,\cdot)\right]$. This completes the proof of Theorem~\ref{t2}.

\end{proof}

Let us now turn to the unique continuation result formulated in Section 2.2. 
Recall that Theorem \ref{t3} follows from a local  Holmgren-John unique continuation.  Consider a smooth surface
$\mathcal T:=\{(t,x)\in \R^{1+n}:\ \psi(t,x)=0\}$. We say that the differential operator $\p_t^2-\Delta_x+q$ is non-characteristic at a point $(t,x)\in\mathcal T$ if the outward unit normal vector $\textbf{n}=(n_0,n')$ with respect to $\p\{(t,x)\in \R^{1+n}:\ \psi(t,x)\leq 0\}$ at $(t,x)$, with $n_0\in\R$, $n'\in\R^n$, satisfies $n_0^2\neq |n'|^2$. For all $r>0$ and all $t\in\R$, $x\in\R^n$ we fix
$$B'(x,r)=\{y\in\R^n:\ |y-x|\leq r\},\quad B((t,x),r)=\{(s,y)\in\R^n:\ |(s-t,y-x)|\leq r\}.$$

\begin{Thm}\label{t4} Let $(t,x)\in\mathcal T$. Assume that there exists  $\delta>0$  such that  $q\in L^\infty(B'(x,\delta))$ and such that $\p_t^2-\Delta_x+q$ is non-characteristic at $\mathcal T\cap B((t,x),\delta)$. Then, if $u\in H^1(B((t,x),\delta))$ solves $(\p_t^2-\Delta_x+q)u=0$ in $B((t,x),\delta)$ and satisfies $\supp(u)\subset \{(s,y)\in  B((t,x),\delta):\ \psi(s,y)\leq0\}$ then $\supp(u)\cap\mathcal T=\emptyset$. \end{Thm}
This theorem is a special case of \cite[Theorem 1]{T1} (see also \cite{RoZu} for related results). We refer also to \cite[Theorem 2.66]{KKL} for a proof  without microlocal analysis.
In order to prove Theorem \ref{t3}, we fix $\Omega_1\supset \Omega$ and we consider two intermediate results.
\begin{lem}\label{l2} Let $x_0\in \R^n$, $\rho_0,\delta_0>0$ and $q\in L^\infty(B'(x_0,\rho_0+\delta_0))$. Assume that there exists $u\in H^1((-\delta_0,\delta_0)\times B'(x_0,\rho_0+\delta_0))$ solving $\p_t^2u-\Delta_x u+qu=0$ in  $(-\delta_0,\delta_0)\times B'(x_0,\rho_0+\delta_0)$ and satisfying
\bel{l2b} u(t,x)=0,\quad (t,x)\in[-\delta_0,\delta_0]\times B'(x_0,\rho_0).\ee
Then, we have
\bel{l2c} u(t,x)=0,\quad (t,x)\in K_{x_0,\delta_0}\ee
with
$$K_{x_0,\delta_0}=\{(t,x)\in[-\delta_0,\delta_0]\times B'(x_0,\rho_0+\delta_0):\ |x-x_0|\leq \delta_0-|t|\}.$$
\end{lem}
\begin{proof}
We start by introducing the set
$$K_s=\{(t,x)\in [-\delta_0,\delta_0]\times\overline{\Omega_1}:\ (t^2+s^2)^{1/2}+|x-x_0|\leq\delta_0\},\quad s\in[0,\delta_0]$$
and we remark that $K_0=K_{x_0,\delta_0}$ and $K_{\delta_0}=\{(0,x_0)\}$. The surface $\p K_s$ are smooth with the exception of the end points $(t,x)=(\pm (\delta_0^2-s^2)^{1/2},x_0)$. In addition, for $s\in(0,\delta_0]$,  the points $\p K_s\setminus\{ (\pm (\delta_0^2-s^2)^{1/2},x_0)\}$ are non-characteristic with respect to $\p_t^2-\Delta_x+q$. 

Indeed, for every $(t,x)\in \p K_s\setminus\{ (\pm (\delta_0^2-s^2)^{1/2},x_0)\}$ the  normal derivative (modulo multiplication by $-1$) $\textbf{n}(t,x)$ of $\p K_s$ is given by
 $$\textbf{n}(t,x)=\left(\frac{t^2+s^2}{2t^2+s^2}\right)^\frac{1}{2}\left( \frac{t}{ (t^2+s^2)^\frac{1}{2}},\frac{x-x_0}{|x-x_0|}\right)$$
and we clearly have
$$\frac{t^2}{t^2+s^2}<1=\abs{\frac{x-x_0}{|x-x_0|}}^2,\quad s\in(0,\delta_0].$$
This proves that, for $s\in(0,\delta_0]$, $\p_t^2-\Delta_x+q$ is non-characteristic at any $(t,x) \in \p K_s\setminus\{ (\pm (\delta_0^2-s^2)^{1/2},x_0)\}$.
We fix
$$s_0=\inf\{s\in [0,\delta_0]:\ u_{|K_s}=0\}.$$
Let us show  ad absurdio  that $s_0=0$. For this purpose let us assume that $s_0>0$. According to \eqref{l2b}, we have $s_0<\delta_0$ and it is clear that for any  $(t,x)\in K_{x_0,\delta_0}$ such that\[ 
(t^2+{s_0}^2)^{1/2}+|x-x_0|<\delta_0
\] 
there exists $s\in(s_0,\delta_0]$ such that $(t,x)\in K_s$. Then, the definition of $s_0$ implies that $u=0$ on a neighborhood of $(t,x)$.  

According to the local unique continuation result of Theorem \ref{t4}, since $\p K_{s_0}\setminus\{ (\pm (\delta_0^2-s_0^2)^{1/2},x_0)\}$ is non-characteristic  and since $\supp(u)\cap\{(t,x)\in K_{x_0,\delta_0}:\ (t^2+{s_0}^2)^{1/2}+|x-x_0|<\delta_0\}=\emptyset$, we deduce that $u=0$ on a neighborhood of every point of $\p K_{s_0}\setminus \left(\R\times \{x_0\}\right)$. Moreover, if $(t,x)\in  \p K_{s_0}\cap \left(\R\times \{x_0\} \right)$ we have $(t,x)=(\pm (\delta_0^2-s^2)^{1/2},x_0)$ and \eqref{l2b} implies that $u=0$ on a neighborhood of $(t,x)$. Thus, there is an open neighborhood $V$ of $K_{s_0}$ such that $u_{|V}=0$. As $\p K_{s_0}$ is compact, there exists $\epsilon\in(0,s_0)$ such that $u_{|K_{s_0-\epsilon}}=0$. This contradicts the definition of $s_0$ and we deduce that $s_0=0$. This completes the proof of the lemma.

\end{proof}

From the previous result we can deduce the following.

\begin{lem}\label{l3} Let $\Omega_1$ be an open set of $\R^n$ and let $q\in L^\infty(\Omega_1)$. Consider $x_0\in \Omega_1$ and $\rho>0$, $\delta\in(0,\tau)$ such that
\bel{l3a} B'(x_0,\rho+\delta)\subset\Omega_1.\ee
 Let $u\in H^1((-\tau,\tau)\times \Omega_1)$ be a solution of $\p_t^2u-\Delta_xu+qu=0$
in $(-\tau,\tau)\times \Omega_1$ satisfying 
\bel{l3b} u(t,x)=0,\quad (t,x)\in[-\tau,\tau]\times B'(x_0,\rho).\ee
Then, we have
\bel{l3c} u(t,x)=0,\quad (t,x)\in K_{x_0,\tau}\ee
with
$$K_{x_0,\tau}=\{(t,x)\in[-\tau,\tau]\times B'(x_0,\rho+\delta):\ |x-x_0|\leq \tau-|t|\}.$$
\end{lem}
\begin{proof} Without loss of generality we assume that $\tau\geq\rho+\delta$. Then our goal is to show that for any $t\in[-\tau,\tau]$ and any $x\in B'(x_0,\rho+\delta)$ such that $|x-x_0|\leq \tau-|t|$ we have $u(t,x)=0$. We divide this proof into two steps.\\
\textbf{First step}: let $t\in(-\tau,\rho+\delta-\tau)\cup(\tau-\delta-\rho,\tau)$ and consider $v$ defined by
$$v(s,x):=u(t+s,x),\quad (s,x)\in[|t|-\tau,\tau-|t|]\times B'(x_0,\delta+\rho).$$
We have $v\in H^1((|t|-\tau,\tau-|t|)\times B'(x_0,\rho+\delta))$ and $v$ is a solution of $\p_s^2v-\Delta_x v+qv=0$ in  $(|t|-\tau,\tau-|t|)\times B'(x_0,\rho+\delta)$. Moreover, since $\tau-|t|\in[0,\rho+\delta)$, fixing $\delta_0=\tau-|t|$ and $\rho_0=\min(\rho+\delta+|t|-\tau,\rho)$, \eqref{l3b} implies that $v=0$ on $[-\delta_0,\delta_0]\times  B'(x_0,\rho_0)$. Therefore, applying Lemma \ref{l2} to $v$, we deduce that
$$v(s,x)=0,\quad (s,x)\in \{(s,x)\in [|t|-\tau,\tau-|t|]\times B'(x_0,\tau-|t|+\rho_0):\ |x-x_0|\leq \tau-|t|-|s|\}.$$
In particular we have
$$u(t,x)=v(0,x)=0,\quad x\in B'(x_0,\tau-|t|).$$
This shows that 
\bel{l3d}u(t,x)=0,\quad (t,x)\in K_{x_0,\tau}\cap [([-\tau,\rho+\delta-\tau)\cup(\tau-\delta-\rho,\tau])\times B'(x_0,\rho+\delta)].\ee
\textbf{Second step}: let $t\in[\rho+\delta-\tau,\tau-\delta-\rho]$ and consider $w$ defined by
$$w(s,x):=u(t+s,x),\quad (s,x)\in[-\rho-\delta,\rho+\delta]\times B'(x_0,\delta+\rho).$$
Then, for any $\delta_0\in[\delta,\rho+\delta)$ and $\rho_0=\rho+\delta-\delta_0$ applying again Lemma \ref{l2}, we deduce that
$$w(s,x)=0,\quad (s,x)\in \{(s,x)\in [-\delta_0,\delta_0]\times B'(x_0,\delta+\rho):\ |x-x_0|\leq \delta_0-|s|\}.$$
In particular, we have
$$u(t,x)=w(0,x)=0,\quad x\in B'(x_0,\delta_0).$$
Since $\delta_0\in [\delta,\rho+\delta)$ is arbitrary, we can send $\delta_0\to\rho+\delta$ and deduce that
$$u(t,x)=0,\quad x\in B'(x_0,\rho+\delta).$$
Therefore, we have
\bel{l3f}u(t,x)=0,\quad (t,x)\in K_{x_0,\tau}\cap ([\rho+\delta-\tau,\tau-\delta-\rho]\times B'(x_0,\rho+\delta)).\ee
Finally, combining \eqref{l3d}-\eqref{l3f}, we deduce \eqref{l3c}.

\end{proof} 

We are now in the position to complete the proof of Theorem \ref{t3}.

\begin{proof}[Proof of Theorem \ref{t3} under the assumption (H)]
Replacing $u(t,x)$ by $u(t+\tau,x)$, we can without loss of generality assume that $(\p_t^2-\Delta_x+q)u=0$ on $(-\tau,\tau)\times\Omega$ and $\p_\nu u=u=0$ on $(-\tau,\tau)\times S$. Then, fixing the cone
$$K_{S,\tau}:=\{(t,x)\in (-\tau,\tau)\times\Omega:\ \textrm{dist}(x,S)< \tau-|t|\}$$
the proof will be completed if we show that
\bel{t3c}u(t,x)=0,\quad (t,x)\in K_{S,\tau}.\ee
Here we use the fact that $\overline{K_{S,\tau}}\cap (\{0\}\times \overline{\Omega})=\{0\}\times\Omega(S,\tau)$.

Fix $(t_0,x_0)\in(-\tau,\tau)\times\Omega$ such that $\textrm{dist}(x_0,S)< \tau-|t_0|$. We consider $\epsilon_1>0$ arbitrary small. Then, according to condition (H), there exists $z_0\in S$ such that $|z_0-x_0|= 3\epsilon_1+ \dist(x_0,S)$, $[x_0,z_0]\subset(\Omega\cup S)$ and $\{y\in\Gamma:\ |y-z_0|<2\epsilon_1\}\subset S$. We define $\Omega_1=\Omega\cup \{x\in\R^n:\ |x-z_0|<2\epsilon_1\}$ and we clearly have $(\p\Omega\setminus \p\Omega_1)\subset S$. By eventually reducing the size of $\epsilon_1$, we can assume that the element $z\in\R^n$ given by
$$z=z_0+\epsilon_1 \frac{z_0-x_0}{|z_0-x_0|}$$
is lying in  $ \Omega_1\setminus \overline{\Omega}$ and  $[x_0,z]\subset\Omega_1$. Moreover, fixing $\epsilon=4\epsilon_1$, which can be arbitrary small since $\epsilon_1>0$ can be arbitrary small, 
we deduce that $|x_0-z|=\ell:=\textrm{dist}(x_0,S)+\epsilon$. We extend $u$ by $0$ to $(-\tau,\tau)\times\Omega_1$. Since  
$\p_\nu u=u=0$ on $(-\tau,\tau)\times S$,  we deduce that $u\in H^1((-\tau,\tau)\times \Omega_1)$. Therefore extending  $q$ to $\overline{\Omega_1}$ we deduce that $u$ solves $\p_t^2-\Delta_x u+qu=0$ in $(-\tau,\tau)\times \Omega_1$. Using the fact that $[x_0,z]\subset\Omega_1$, we consider the path
$$\mu(s)= \frac{sx_0+(|x_0-z|-s)z}{|x_0-z|} ,\quad s\in[0,|x_0-z|]$$
which is lying in $\Omega_1$. Since $\mu([0,\ell])$ is compact there exists $\delta>0$ such that dist$(\mu([0,\ell]),\p\Omega_1)>2\delta$ and dist$(\mu(0),\p\Omega)>2\delta$. Now let $N\in \mathbb N$ be such that $N\delta=\ell+\epsilon$. Here we can eventually reduce the size of $\delta$ in order to have $\frac{\ell+\epsilon}{\delta}\in\mathbb N$. Choose $s_j\in[0,\ell]$, $j=0,\ldots,N,$ such that $0<s_{j+1}-s_j<\delta$, $s_0=0$, $s_N=\ell$ and denote $y_j=\mu(s_j)$. We can now apply Lemma \ref{l3} to complete the proof of the theorem. Indeed, we can choose $\rho\in(0,\delta)$, which can be arbitrary small,  such that $u=0$ on $(-\tau,\tau)\times B'(y_0,\rho)$. Hence, Lemma \ref{l3} implies that $u=0$ in 
\[K_{y_0,\tau}=\{(t,x)\in(-\tau,\tau)\times B(y_0,\delta+\rho):\ |x-y_0|\leq\tau-|t|\}.\]
On the other hand, using the fact that $|y_1-y_0|<\delta$,  for all $(t,x)\in[-\tau+|y_1-y_0|+\rho, \tau-|y_1-y_0|-\rho]\times B'(y_1,\rho)$, we have
$$|x-y_0|\leq |x-y_1|+|y_1-y_0|<|y_1-y_0|+\rho< \min(\tau-|t|,\delta+\rho).$$
Therefore, using the fact that $s_1=|y_0-y_1|$, we find 
$$[-\tau+s_1+\rho, \tau-s_1-\rho]\times B'(y_1,\rho)= [-\tau+|y_1-y_0|+\rho, \tau-|y_1-y_0|-\rho]\times B'(y_1,\rho)\subset K_{y_0,\tau}$$
and
$$u(t,x)=0,\quad (t,x)\in [-\tau+s_1+\rho, \tau-s_1-\rho]\times B'(y_1,\rho).$$
Repeating this process and by eventually reducing the size of $\rho$,  we find
$$u(t,x)=0,\quad (t,x)\in [-\tau+s_j+j\rho, \tau-s_j-j\rho]\times B'(y_j,\rho),\quad j=0,\ldots,N.$$
Note that here we use the fact that
$$|y_{j+1}-y_{j}|+s_j=s_{j+1},\quad j=0,\ldots,N-1.$$
Using the fact that $s_N=\ell=\textrm{dist}(x_0,S)+\epsilon$, we get
$$[-\tau+s_N+N\rho, \tau-s_N-N\rho]\times B'(y_N,\rho)=[-\tau+\ell+N\rho, \tau-\ell-N\rho]\times B'(x_0,\rho).$$
Therefore, using the fact that dist$(x_0,S)<\tau-|t_0|$ and the fact that $\epsilon$ and $\rho$ are arbitrary and $N$ is independent of $\rho$, we can choose $\epsilon$ and $\rho$ in  such a way that $\tau-|t_0|-$dist$(x_0,S)> 2\epsilon+N\rho$. It follows that
$$[t_0-\epsilon, t_0+\epsilon]\times B'(x_0,\rho)\subset[-\tau+s_N+N\rho, \tau-s_N-N\rho]\times B'(y_N,\rho)$$
and we have
$$u(t,x)=0,\quad (t,x)\in [t_0-\epsilon, t_0+\epsilon]\times B'(x_0,\rho).$$
 This proves \eqref{t3c} and Theorem \ref{t3}.
 \end{proof}

The rest of the appendix concerns the proofs of the two approximate controllability results that we need.

\begin{proof}[Proof of Corollary \ref{c1}]
In order to prove the density result we fix $h\in L^2(\Omega(\gamma',\tau))$, extended by zero to $h\in L^2(\Omega)$, such that 
\bel{t6b}\left\langle u_j^f(\tau,\cdot),h\right\rangle_{L^2(\Omega)}= \left\langle u_j^f(\tau,\cdot),h\right\rangle_{L^2(\Omega(\gamma',\tau))}=\int_{\Omega(\gamma',\tau)}u_j^f(\tau,\cdot)\overline{h}dx=0,\quad f\in \mathcal C^\infty_0((0,T)\times\gamma'),\ t\in[0,\tau]\ee
and we will prove that $h=0$. For this purpose, let $e_j\in\mathcal C([0,\tau];H^1(\Omega))\cap \mathcal C^1([0,\tau];L^2(\Omega))$ solves 
\begin{equation}\label{eq2}\left\{\begin{array}{ll}\partial_t^2e_j-\Delta_x e_j+q_j(x)e_j=0,\quad &\textrm{in}\ (0,\tau)\times\Omega,\\  e_j(\tau,\cdot)=0,\quad \partial_te_j(\tau,\cdot)=h,\quad &\textrm{in}\ \Omega,\\ e_j=0,\quad &\textrm{on}\ (0,\tau)\times\p\Omega.\end{array}\right.\end{equation}
In light of \cite[,Theorem 2.1]{LLT}, we have ${\p_\nu e_j}_{|(0,\tau)\times\p\Omega}\in L^2((0,\tau)\times\p\Omega)$.
Thus, for all $f\in\mathcal C^\infty_0((0,T)\times\gamma')$,  integrating by parts and applying \eqref{eq1} and \eqref{t6b}, we find
$$\begin{aligned}0
&=\int_0^\tau\int_\Omega u_j^f\overline{(\partial_t^2e_j-\Delta_x e_j+q_j(x)e_j)}dxdt\\
\ &=\int_\Omega u^f_j(\tau,\cdot)\overline{h}dx+\int_0^\tau \int_{\p\Omega} f(t,x)\overline{\p_\nu e_j}d\sigma(x)dt 
+\int_0^\tau\int_\Omega (\partial_t^2u_j^f-\Delta_xu^f_j+q_j(x)u_j^f)\overline{e_j}dxdt
\\
\ &=\int_0^\tau \int_{\gamma'} f(t,x)\overline{\p_\nu e_j}d\sigma(x)dt.\end{aligned}$$
Allowing $f\in\mathcal C^\infty_0((0,T)\times\gamma')$ be arbitrary we deduce that ${\p_\nu e_j}_{|(0,\tau)\times\gamma'}=0$. Now fixing $E_j$ defined by
$$E_j(t,x):=\left\{\begin{array}{ll} e_j(t,x)\quad &\textrm{for }t\leq\tau,\\ e_j(2\tau-t,x) &\textrm{for }t>\tau,\end{array}\right.$$
we deduce that $E_j\in H^1((0,2\tau)\times \Omega)$ satisfies
$$\left\{\begin{array}{ll}\partial_t^2E_j-\Delta_x E_j+q_j(x)E_j=0,\quad &\textrm{in}\ (0,2\tau)\times\Omega,\\ E_j=\p_\nu E_j=0,\quad &\textrm{on}\ (0,2\tau)\times\gamma'.\end{array}\right.$$
Thus, in view of Theorem \ref{t3}, we have
$h= \partial_t E_j(\tau,\cdot)_{|\Omega(\gamma',\tau)}=0$. This proves the density of \eqref{c1a}. 

\end{proof}

\begin{proof}[Proof of Lemma \ref{ll1}]
 In order to prove the density result we fix $h\in L^2(\Omega(B,r))$, extended by zero to $h\in L^2(\Omega)$, such that 
\bel{t6bb}\left\langle v_{j,F}(T,\cdot),h\right\rangle_{L^2(\Omega)}= \left\langle v_{j,F}(T,\cdot),h\right\rangle_{L^2(\Omega(B,r))}=\int_{\Omega(\gamma',\tau)}v_{j,F}(T,\cdot)\overline{h}dx=0,\quad F\in \mathcal C^\infty_0((T-r,T)\times B)\ee
and we will prove that $h=0$. For this purpose, let $e_j\in\mathcal C([0,\tau];H^1(\Omega))\cap \mathcal C^1([0,\tau];L^2(\Omega))$ solves 
\begin{equation}\label{eq2b}\left\{\begin{array}{ll}\partial_t^2e_j-\Delta_x e_j+q_j(x)e_j=0,\quad &\textrm{in}\ Q,\\  e_j(T,\cdot)=0,\quad \partial_te_j(T,\cdot)=h,\quad &\textrm{in}\ \Omega,\\ e_j=0,\quad &\textrm{on}\ (0,T)\times\p\Omega\end{array}\right.\end{equation}
Then, for all $F\in\mathcal C^\infty_0((T-r,T)\times B)$,  integrating by parts and applying \eqref{t6bb}, we find
$$\begin{aligned}0
&=\int_0^T\int_\Omega v_{j,F}\overline{(\partial_t^2e_j-\Delta_x e_j+q_j(x)e_j)}dxdt
\\
\ &=\int_\Omega v_{j,F}(T,\cdot)\overline{h}dx-\int_0^T \int_{\Omega} F(t,x) \overline{e_j}dxdt
+ \int_0^T\int_\Omega (\partial_t^2v_{j,F}-\Delta_xv_{j,F}+q_j(x)v_{j,F})\overline{e_j} dxdt
\\
\ &=-\int_0^T \int_{\Omega} F(t,x) \overline{e_j}dxdt.\end{aligned}$$
Allowing $F\in \mathcal C^\infty_0((T-r,T)\times B)$ be arbitrary we deduce that $ {e_j}_{|(T-r,T)\times B}=0$. Now fixing $E_j$ defined by
$$E_j(t,x):=\left\{\begin{array}{ll} e_j(t,x)\quad &\textrm{for }t\in[0,T),\\ e_j(2T-t,x) &\textrm{for }t\in[T,2T),\end{array}\right.$$
we deduce that $E_j\in H^1((0,2T)\times \Omega)$ satisfies
$$\left\{\begin{array}{ll}\partial_t^2E_j-\Delta_x E_j+q_j(x)E_j=0,\quad &\textrm{in}\ (0,2T)\times\Omega,\\ E_j=0,\quad &\textrm{on}\ (T-r,T+r)\times B.\end{array}\right.$$
Thus, in view of the proof of Theorem \ref{t3}, we have
$h=E_j(T,\cdot)_{|\Omega(B,r)}=0$. This proves the density of \eqref{ll1a}. 

\end{proof}

\end{document}